\documentclass[a4paper]{article}
\usepackage{main}

\title{Accelerated Bregman Proximal Gradient Methods from Dual Geometric Perspectives}

\author{
    Yuya Yamashita\thanks{Graduate School of Information Science and Technology, The University of Tokyo, Tokyo, Japan (\href{mailto:yuya-yamashita@g.ecc.u-tokyo.ac.jp}{\texttt{yuya-yamashita@g.ecc.u-tokyo.ac.jp}})}~\footnotemark[4]
    \and Shota Takahashi\thanks{Graduate School of Information Science and Technology, The University of Tokyo, Tokyo, Japan (\href{mailto:shota@mist.i.u-tokyo.ac.jp}{\texttt{shota@mist.i.u-tokyo.ac.jp}})}
    \and Akiko Takeda\thanks{Graduate School of Information Science and Technology, The University of Tokyo, Tokyo, Japan (\href{mailto:takeda@mist.i.u-tokyo.ac.jp}{\texttt{takeda@mist.i.u-tokyo.ac.jp}})}~\thanks{Center for Advanced Intelligence Project, RIKEN, Tokyo, Japan}
}
\date{\today}
\begin{document}

\maketitle

\begin{abstract}
    We study Bregman proximal gradient (BPG) algorithms under relative smoothness for convex, relatively strongly convex, and nonconvex objectives.
    Existing accelerated BPG algorithms for convex objectives typically require additional assumptions on Bregman divergences, most notably triangle-scaling conditions, which can lead to slower convergence rates.
    We propose a family of geometry-accelerated BPG algorithms that exploit Bregman geometry in both proximal-gradient and mirror-space updates, without imposing a triangle-scaling condition.
    Our methods adapt the stepsizes and mirror-space updates through local backtracking and computable acceptance criteria, without a global relative-smoothness constant as input.
    We derive convergence bounds in terms of the parameters accepted during the iterations.
    These bounds yield an $\O(k^{-2})$ rate for convex objectives and a linear rate for relatively strongly convex objectives when the geometry-acceleration parameters remain uniformly bounded.
    For nonconvex objectives, we establish an $\O(k^{-1})$ rate for a stationarity measure without requiring a lower Bregman bound or a full-domain Bregman divergence.
    Numerical experiments on inverse problems, entropy-regularized least squares, D-optimal design, and nonnegative matrix factorization demonstrate faster practical convergence than established Bregman baselines.
\end{abstract}

\section{Introduction}
\label{sec:introduction}

We focus on minimizing the sum of a smooth function and a nonsmooth convex function:
\begin{equation}
    \min_{x\in\cl\Omega}\quad \Phi(x)\coloneq f(x)+\rho(x), \label{eq:problem}
\end{equation}
where $\cl\Omega$ is the closure of a nonempty open convex set $\Omega\subset\R^n$, $f:\R^n\to(-\infty,+\infty]$ is proper, lower semicontinuous, and continuously differentiable on $\Omega$, and $\rho:\R^n\to(-\infty,+\infty]$ is proper, lower semicontinuous, and convex.
Relative smoothness~\citep{bauschke2017descent,bolte2018first,lu2018relatively} replaces the quadratic upper approximation of $f$ with an upper approximation based on a Bregman divergence.
Under this assumption, we study convex, relatively strongly convex, and nonconvex $f$.

\paragraph{Related work.}
Classical first-order methods, including proximal gradient methods, are typically analyzed under Euclidean smoothness (see~\citep{Beck2017-qc,Nesterov2018-gj}), \ie, $\nabla f$ is Lipschitz continuous.
This assumption provides a simple quadratic upper bound on the objective and enables accelerated algorithms with fast convergence rates~\citep{nesterov1983method,beck2009fast}.
However, some objectives fall outside this setting because their curvature is unbounded over the effective domain, for example near its boundary.
Bregman geometry, which also underlies mirror descent~\citep{nemirovski1983problem}, allows upper bounds better adapted to objective curvature, providing a basis for Bregman proximal gradient (BPG) algorithms~\citep{bauschke2017descent,bolte2018first,lu2018relatively}.
BPG algorithms and their variants yield convergence guarantees without global Euclidean smoothness, with applications to Poisson linear inverse problems~\citep{bauschke2017descent} and nonconvex inverse problems such as phase retrieval \citep{bolte2018first,takahashi2022new}, blind deconvolution~\citep{Takahashi2023-uh}, and nonnegative matrix factorization~\citep{Mukkamala2019-mk,Takahashi2026-rv}.
See also Appendix~\ref{app:related-work} for other methods based on Bregman divergences.

\paragraph{Acceleration of BPG.}
In Euclidean geometry, accelerated proximal gradient algorithms attain the $\O(k^{-2})$ convergence rate after $k$ iterations for convex objectives.
Extending such acceleration to Bregman geometry is substantially more difficult.
In fact, under relative smoothness alone, the $\O(k^{-1})$ rate is optimal for a broad class of Bregman first-order algorithms \citep{dragomir2022optimal}.
Therefore, achieving a faster global rate for this broad class requires additional structure beyond relative smoothness.
\citet{hanzely2021accelerated} obtain an $\O(k^{-\gamma})$ rate by imposing a triangle-scaling inequality on the Bregman divergence, where the triangle-scaling exponent $\gamma\in(0,2]$ controls how the divergence scales under linear coupling.
Many subsequent accelerated Bregman algorithms have continued to rely on triangle-scaling conditions~\citep{liu2022dual, savchuk2025accelerated}.
On the other hand, the triangle-scaling exponent can deteriorate substantially for some Bregman divergences; in particular, \citet{hanzely2021accelerated} report that no uniform exponent $\gamma > 1/2$ is available for the Itakura--Saito divergence.
\citet{hanzely2021accelerated} also develop ABPG-g, which adapts gains using local triangle-scaling tests and bounds the objective gap in terms of the intrinsic exponent $\gamma$ and the accepted gains.
Then, the following question naturally arises:

\begin{quote}
    \emph{Can we establish accelerated convergence rates under conditions on the parameters chosen during the iterations, without imposing a global triangle-scaling condition on the Bregman divergence?}
\end{quote}

\paragraph{Contributions.}
We answer this question affirmatively by introducing geometry-accelerated BPG (GA-BPG).
For convex objectives, its acceptance test controls the auxiliary mirror displacement instead of triangle scaling.
We develop three variants, GA-BPGc, GA-BPGsc, and GA-BPGnc, for convex, relatively strongly convex, and nonconvex objectives, respectively.
Our main contributions are summarized as follows:
\begin{itemize}
    \item \textbf{Mirror geometry.} GA-BPG constructs auxiliary points in mirror space to guide subsequent BPG steps.
    Computable criteria determine whether to accept these updates, adjust their parameters, or reset the auxiliary points.
    \item \textbf{Acceleration.} We derive convergence bounds in terms of the parameters accepted during the iterations.
    For convex and relatively strongly convex objectives, these bounds yield $\O(k^{-2})$ and linear convergence, respectively, when the geometry-acceleration parameters remain uniformly bounded.
    Moreover, we also establish sufficient conditions for geometry parameters to remain uniformly bounded.
    For nonconvex objectives, we obtain an $\O(k^{-1})$ bound on a stationarity residual.
    \item \textbf{Minimal assumptions.} Under the standard BPG assumptions for each regime, we derive convergence bounds that depend on the iterates and parameters generated by the algorithms.
    GA-BPGc requires neither a triangle-scaling condition nor global curvature constants as input.
    GA-BPGnc requires neither a lower Bregman bound nor the assumption $\Omega=\R^n$, relying only on upper relative smoothness and lower boundedness of the objective.
    None of the three algorithms requires a global relative-smoothness constant $L$ as input; GA-BPGsc uses the known relative strong-convexity constant $\mu$.
    \item \textbf{Numerical evaluation.} Experiments on inverse problems, entropy-regularized least squares, D-optimal design, and nonnegative matrix factorization demonstrate faster practical convergence than established Bregman baselines.
\end{itemize}

Table~\ref{tab:comparison} compares our convergence guarantees with existing results.
GA-BPGc and GA-BPGsc attain accelerated rates if the accepted geometry-acceleration parameters are uniformly bounded.

\begin{table}[t]
    \centering
    \caption{
        \textbf{Assumptions and convergence guarantees.}
        “Cond.”: additional conditions; $\mathrm{TS}_{\gamma}$: local triangle-scaling with gains and intrinsic exponent $\gamma\in(0,2]$.
        “No $L$ input”: the global relative-smoothness constant need not be supplied.
        Rates are for objective gaps in convex (cvx) and relatively strongly convex (r-scvx) problems, and stationarity measures in nonconvex (ncvx) problems.
        $\dagger$: rates hold with uniformly bounded geometry-acceleration parameters; sufficient conditions are given in~\cref{thm:convex-kernel-certificate,thm:sc-acceleration-conditions}.
        $\flat$: the bound depends on accepted gains; $C^2$ Bregman generators allow $\gamma=2$, yielding $\O(k^{-2})$ with bounded geometric-mean gains.
    }
    \label{tab:comparison}
    \begin{tabular}{lccccc}
        \toprule
        \textbf{Method}
        & Reference
        & $f$
        & Cond.
        & No $L$ input
        & Rate \\
        \midrule
        BPG
        & \citet{bauschke2017descent}
        & cvx
        & -
        & \xmark
        & $\O(k^{-1})$
        \\

        ABPG-g
        & \citet{hanzely2021accelerated}
        & cvx
        & $\mathrm{TS}_{\gamma}$
        & \xmark
        & $\O(k^{-\gamma})^{\flat}$
        \\

        \textbf{GA-BPGc}
        & Algorithm~\ref{alg:convex-full}
        & cvx
        & -
        & \cmark
        & $\O(k^{-2})^{\dagger}$
        \\
        \midrule
        BPG
        & \cite{bauschke2019linear}
        & r-scvx
        & -
        & \xmark
        & linear
        \\

        \textbf{GA-BPGsc}
        & Algorithm~\ref{alg:strong-full}
        & r-scvx
        & -
        & \cmark
        & $\mathrm{linear}^{\dagger}$
        \\
        \midrule
        BPG
        & \citet{bolte2018first}
        & ncvx
        & -
        & \xmark
        & $\O(k^{-1})$
        \\

        BPGe
        & \citet{zhang2019bregman}
        & ncvx
        & $\Omega = \R^n$
        & \xmark
        & $\O(k^{-1})$
        \\

        \textbf{GA-BPGnc}
        & Algorithm~\ref{alg:nonconvex-full}
        & ncvx
        & -
        & \cmark
        & $\O(k^{-1})$
        \\
        \bottomrule
    \end{tabular}
\end{table}

\paragraph{Notation.}
For vectors, $\langle\cdot,\cdot\rangle$ and $\|\cdot\|$ denote the Euclidean inner product and norm, respectively.
For a matrix $A\in\R^{m\times n}$, $\|A\|_{\F}$ denotes the Frobenius norm.
$\mathbf{1}$ denotes the all-ones vector of appropriate dimension.
Vector and matrix inequalities using $\geq$ or $\leq$ are interpreted componentwise.
For a set $S\subset\R^n$, $\cl S$ and $\interior S$ denote its closure and interior.
The indicator function $\delta_S$ equals $0$ on $S$ and $+\infty$ elsewhere.
For an extended-real-valued function $\psi:\R^n\to[-\infty,+\infty]$, its \emph{effective domain} is $\dom \psi\coloneq \{x\in\R^n:\psi(x)<+\infty\}$, and its \emph{convex conjugate} is $\psi^*(s)\coloneq \sup_{x\in\R^n}\{\langle s,x\rangle-\psi(x)\}$.
For convex $\psi$, $\partial \psi(x)$ denotes its subdifferential, whose elements are \emph{subgradients}.
The function $\psi$ is proper if $\dom \psi\neq\emptyset$ and $\psi(x)>-\infty$ for all $x\in\R^n$, and a proper lower semicontinuous convex function is \emph{Legendre} if it is essentially smooth and essentially strictly convex.
For a Legendre function $\psi$, the \emph{inverse mirror map} is $\nabla\psi^*=(\nabla\psi)^{-1}$, with $\nabla\psi^*(s)=\argmin_{x\in\R^n}\{\psi(x)-\langle s,x\rangle\}$ for $s\in\interior\dom\psi^*$.
For $u\in\dom\psi$ and $v\in\interior\dom\psi$, define the \emph{Bregman divergence} $D_\psi(u,v)\coloneq \psi(u)-\psi(v)-\langle\nabla\psi(v),u-v\rangle$.
We use the same notation throughout, including for nonconvex $f$.
For each $v\in\interior\dom\psi$, we extend $D_\psi(u,v)$ to $u\in\cl\dom\psi$ by the same formula, allowing the value $+\infty$.
Finally, for $x\in\Omega$, define $\partial\Phi(x)\coloneq \nabla f(x)+\partial\rho(x)$, and call $x$ stationary if $0\in\partial\Phi(x)$.

\section{Proposed algorithm: geometry-accelerated BPG}
\label{sec:setting-methods}

We impose the standing assumptions stated below.
\begin{assumption}[Standing assumptions]
    \label[assumption]{ass:standing}
    The following conditions hold.
    \begin{enumerate}
        \item $\psi:\R^n\to(-\infty,+\infty]$ is a Legendre function with $\Omega = \interior\dom\psi$.
        \item $f:\R^n\to(-\infty,+\infty]$ is proper and lower semicontinuous with $\dom\psi \subset\dom f$, and is continuously differentiable on $\Omega$.
        \item $\rho:\R^n\to(-\infty,+\infty]$ is proper, lower semicontinuous, and convex with $\Omega\cap\dom\rho \neq \emptyset$.
        \item \textbf{Relative smoothness.} \citep{bauschke2017descent, lu2018relatively} $f$ is \emph{relatively smooth} with respect to $\psi$, \ie, there exists $L>0$ such that $D_f(u,v) \leq L D_\psi(u,v)$ holds for all $u,v\in\Omega$.
        \item The objective is bounded below, \ie, $\Phi^\star \coloneq \inf_{x \in \cl \Omega} \Phi(x) > -\infty$.
    \end{enumerate}
\end{assumption}
\cref{ass:standing} is standard in first-order optimization based on Bregman divergences~\citep{bolte2018first,hanzely2021accelerated,takahashi2022new,Takahashi2026-rv}.
Relative smoothness in \cref{ass:standing}(iv) ensures convergence of BPG algorithms and is also called the smooth adaptable property~\citep{bolte2018first}.
Examples of commonly used BPG geometries $\psi$ are listed in \Cref{tab:convex-kernel-conditions}.

\paragraph{Geometry-accelerated BPG.}
Gradient descent and proximal gradient descent use squared Euclidean distances to construct subproblems, whereas mirror descent and BPG use Bregman divergences.
However, BPG can converge slowly on~\eqref{eq:problem} in practice.
\citet{dragomir2022optimal} prove $\O(1/k)$ worst-case optimal for their Bregman first-order methods under convexity and relative smoothness alone.
\citet{hanzely2021accelerated} obtain an $\O(1/k^\gamma)$ rate under a uniform triangle-scaling condition with exponent $\gamma\in(0,2]$.
The exponent depends on the Bregman divergence and can be far below $2$, giving slower convergence rates.
For the Itakura--Saito divergence, no positive uniform exponent exists on its effective domain.
ABPG-g adapts gains using local triangle-scaling tests, with a bound determined by the intrinsic exponent $\gamma$ and the accepted gains~\citep{hanzely2021accelerated}.
Our convex methods use acceptance tests on auxiliary mirror displacements to obtain accelerated bounds.

We now present the common template for the three \emph{geometry-accelerated BPG algorithms} (GA-BPG) in~\cref{alg:common-template}: GA-BPGc (\cref{alg:convex-full}) for convex objectives, GA-BPGsc (\cref{alg:strong-full}) for relatively strongly convex objectives, and GA-BPGnc (\cref{alg:nonconvex-full}) for nonconvex objectives.
Each consists of four components: (i) coupling, (ii) BPG update, (iii) criterion, and (iv) mirror update, whose order may vary.
We explain this structure through GA-BPGc; GA-BPGsc and GA-BPGnc are discussed separately in \cref{subsec:strong,subsec:nonconvex}, respectively.

\begin{algorithm}[t]
    \caption{Common GA-BPG template}
    \label{alg:common-template}
    \KwIn{A Legendre function $\phi$; $y_0\in\Omega\cap\dom\rho\cap\interior\dom\phi$; $\lambda_0>0$; $\gamma_+ \geq 1$, $\gamma_->1$.
    }
    $z_0 \gets y_0$\;
    \For{$k=1,2,\ldots$}{
        Initialize $\lambda\gets\gamma_+\lambda_{k-1}$.\;
        \textbf{Coupling.}
        \label{step:coupling}
        Form the coupled iterate $x$ using $y_{k-1}$ and $z_{k-1}$.\;

        \textbf{BPG update.}
        \label{step:bpg-update}
        Compute
        \begin{equation}
            \widehat y \in T_\lambda(x)
            \coloneq
            \argmin_{u\in\cl \Omega}
            \left\{
            \rho(u)
            +\langle\nabla f(x),u-x\rangle
            +\frac{1}{\lambda}D_\psi(u,x)
            \right\},\label{eq:bpg-subproblem}
        \end{equation}
        \If{$D_f(\widehat y,x)>\lambda^{-1}D_\psi(\widehat y,x)$}{
            $\lambda \gets \lambda/\gamma_-$ and retry.\;
        }

        \textbf{Mirror update.}
        \label{step:mirror-update}
        Construct an auxiliary point $\widehat z$ via an inverse mirror map $\nabla \phi^{\star}$.\;

        \textbf{Criterion.}
        \label{step:criterion}
        Select $y_+$ from $\widehat{y}$ or $y_{k-1}$, and determine whether to accept, retry, or reset $\widehat z$.\;

        $(x_k,y_k,z_k,\lambda_k)\gets(x,y_+,\widehat z,\lambda)$\;
    }
\end{algorithm}

\paragraph{Coupling.}
Let $\lambda>0$ denote the BPG stepsize in~\eqref{eq:bpg-subproblem} and let $\kappa>0$ denote the geometry-acceleration parameter.
GA-BPGc and GA-BPGsc combine $y_{k-1}$ with the auxiliary point $z_{k-1}$:
\begin{equation*}
    \omega = \omega_{k-1} + \alpha, \qquad x = \frac{\omega_{k-1}y_{k-1} + \alpha z_{k-1}}{\omega},
\end{equation*}
where $\omega_k$ is the cumulative weight ($\omega_0=0$).
GA-BPGc backtracks $\kappa$ until acceptance, adjusting the coupling weight and mirror update scale via $\kappa\alpha^2=2\lambda\omega$ ($\alpha>0$).

\paragraph{BPG update.}
All GA-BPG algorithms compute $\widehat y\in T_\lambda(x)$ by solving~\eqref{eq:bpg-subproblem}, where $x$ is the point obtained from the coupling step.
We make the following assumption for the subproblem.
\begin{assumption}\label[assumption]{assumption:subproblem}
    For every $x \in \Omega$ and $\lambda > 0$, subproblem~\eqref{eq:bpg-subproblem} has a minimizer.
\end{assumption}
\cref{assumption:subproblem} holds, for example, if $\psi+\lambda\rho$ is supercoercive for every $\lambda>0$~\citep[Lemma~2(ii)]{bauschke2017descent}.
The solution set $T_\lambda(x)$ is a singleton contained in $\Omega$; see \cref{prop:bpg-unique}.
All GA-BPG algorithms also use backtracking until $\widehat y$ satisfies the relative smoothness inequality
\begin{equation}
    D_f(\widehat y,x)\leq\lambda^{-1}D_\psi(\widehat y,x). \label{eq:local-rs-test}
\end{equation}
The complexity of backtracking is discussed in Appendix~\ref{app:common-backtracking}.

\paragraph{Mirror update.}
Let $g_\lambda(x)\coloneq(\nabla\psi(x)-\nabla\psi(\widehat y))/\lambda$, where $\widehat y\in T_\lambda(x)$.
GA-BPGc uses a Legendre function $\phi:\R^n\to(-\infty,+\infty]$, possibly different from $\psi$, and sets
\begin{equation*}
    \widehat z = \nabla \phi^{\ast}(\nabla \phi(z_{k-1}) - \alpha g_\lambda(x)).
\end{equation*}
The use of a separate auxiliary mirror geometry is different from ABPG-g~\citep{hanzely2021accelerated}, which forms its next primal iterate by linear interpolation after a Bregman proximal step.
The auxiliary mirror update can use nonlinear geometric structure that primal-space interpolation may not fully capture.
Practical guidance for choosing $\phi$ is given after Theorem~\ref{thm:convex-posterior}, and general sufficient conditions for accelerated convergence are established in Theorem~\ref{thm:convex-kernel-certificate}.
Representative geometry pairs are listed in Table~\ref{tab:convex-kernel-conditions}, with proofs and further discussion in Appendix~\ref{app:convex-sublevel-results}.

\paragraph{Criterion.}
Unlike the other three components, this component handles auxiliary calculations, domain checks, and acceptance decisions at multiple points in each algorithm.
For GA-BPGc and GA-BPGsc, the acceptance inequality keeps the Lyapunov function nonincreasing.
A failed mirror-domain or acceptance check increases $\kappa$ and triggers a new trial; see~\cref{sec:convergence} for details.
On acceptance, both set $y_k=\widehat y$ and record the accepted geometry-acceleration parameter $\kappa_k=\kappa$.

\paragraph{Termination.}
Stopping tests are omitted from the pseudocode.
GA-BPGc and GA-BPGsc return $y_{k-1}$ before the inner search if $y_{k-1}\in T_{\lambda_{k-1}}(y_{k-1})$.
All three algorithms return $x$ whenever a BPG trial gives $\widehat y=x$.
The numerical stopping rules used in the experiments are described in \cref{sec:experiments}.

\begin{algorithm}[t]
    \caption{GA-BPGc}
    \label{alg:convex-full}
    \SetKwFunction{BT}{Backtracking}

    \KwIn{
        A Legendre function $\phi$;
        $y_0\in\Omega\cap\dom\rho\cap\interior\dom\phi$;
        $\lambda_0,\kappa_0>0$;
        $\gamma_+\geq1$; $\gamma_-,\gamma_\kappa>1$.
    }

    $z_0\gets y_0$, $\omega_0\gets0$\;

    \For{$k=1,2,\ldots$}{
    $(\widehat\lambda,\widehat\kappa)\gets(\gamma_+\lambda_{k-1},\kappa_{k-1}/\gamma_+)$\;
    $(\alpha_k,\omega_k,x_k,y_k,z_k,\lambda_k,\kappa_k)\gets$
    \BT{$\omega_{k-1},y_{k-1},z_{k-1},\widehat\lambda,\widehat\kappa$}\;
    }

    \Procedure{\BT{$\omega_-,y,z,\lambda,\kappa$}}{
        $\alpha\gets
            (\lambda+\sqrt{\lambda^2+2\kappa\lambda \omega_-})/\kappa$,
        $\omega\gets \omega_-+\alpha$,
        $x\gets(\omega_-y+\alpha z)/\omega$\label{line:coupling-convex}\Comment*{Coupling}

        $\widehat y\in T_\lambda(x)$
        \Comment*{BPG update}

        \If{$D_f(\widehat y,x)>D_\psi(\widehat y,x)/\lambda$}{
            \Return{\BT{$\omega_-,y,z,\lambda/\gamma_-,\kappa$}}\;
        }

        $p\gets \nabla\phi(z)
            -\alpha(\nabla\psi(x)-\nabla\psi(\widehat y))/\lambda$\;

        \If{$p\notin\interior\dom\phi^*$}{
            \Return{\BT{$\omega_-,y,z,\lambda,\gamma_\kappa\kappa$}}\;
        }

        $\widehat z\gets\nabla\phi^*(p)$
        \Comment*{Mirror update}

        \If(\Comment*[f]{Criterion}){
            $\widehat z\notin\Omega$ \textbf{or}
            $D_\phi(z,\widehat z)
                >\omega P_\lambda(x)+\omega_-D_f(y,x)$
        }{
            \Return{\BT{$\omega_-,y,z,\lambda,\gamma_\kappa\kappa$}}\;
        }

        \Return{$(\alpha,\omega,x,\widehat y,\widehat z,\lambda,\kappa)$}\;
    }
\end{algorithm}

\section{Convergence analysis}
\label{sec:convergence}

This section gives complexity bounds for all three algorithms.
For the unique $\widehat y\in T_\lambda(x)$, define
\begin{equation*}
    P_\lambda(x)\coloneq
    \frac{D_\psi(x,\widehat y)+D_\psi(\widehat y,x)}{\lambda}
    -D_f(\widehat y,x).
\end{equation*}
The following comparison inequality explains the role of $P_\lambda$.
Its proof is given in Appendix~\ref{app:common-comparison}.
\begin{lemma}[BPG comparison]\label[lemma]{lem:bpg-comparison}
    Suppose that \cref{ass:standing,assumption:subproblem} hold.
    Let $\widehat y\in T_\lambda(x)$, with $x\in\Omega$ and $\lambda>0$.
    For every $u\in \cl\Omega\cap\dom \Phi$,
    \begin{equation}
        \Phi(\widehat y)-\Phi(u)
        \leq
        \langle g_\lambda(x),x-u\rangle
        -P_\lambda(x)-D_f(u,x).
        \label{eq:bpg-comparison}
    \end{equation}
    If~\eqref{eq:local-rs-test} holds, then $P_\lambda(x)\geq\lambda^{-1}D_\psi(x,\widehat y)\geq0$, with equality $P_\lambda(x)=0$ if and only if $x=\widehat y$.
\end{lemma}
In \Cref{subsec:convex,subsec:strong}, $u$ is a comparison point used only in the analysis.
Taking $u=x^\star$ yields optimality-gap bounds whenever $x^\star$ lies in the comparison set.

\subsection{Convex objectives}\label{subsec:convex}
We have a complexity bound for GA-BPGc for convex objective functions.
\begin{assumption}[Convexity]\label[assumption]{assumption:convex}
    The function $f$ is convex.
\end{assumption}

Take $u\in\mathcal U_\phi$, where $\mathcal U_\phi\coloneq \cl\Omega\cap\dom \Phi\cap\dom \phi$, and define the Lyapunov function as follows:
\begin{equation*}
    \mathcal E_k^{\mathrm C}(u)
    \coloneq
    \omega_k(\Phi(y_k)-\Phi(u))+D_\phi(u,z_k).
\end{equation*}
Combining two applications of \cref{lem:bpg-comparison} with the coupling and mirror identities in Appendix~\ref{app:common-identities} yields the following estimate; see Appendix~\ref{app:convex-one-step-proof} for the proof.

\begin{lemma}[Convex one-step estimate]
    \label{lem:convex-one-step}
    Suppose that \cref{ass:standing,assumption:subproblem,assumption:convex} hold.
    For every accepted iteration $k\geq1$ of GA-BPGc and every $u\in\mathcal U_\phi$,
    \begin{equation}
        \mathcal E_k^{\mathrm C}(u)
        -\mathcal E_{k-1}^{\mathrm C}(u)
        \leq
        D_\phi(z_{k-1},z_k)
        -\omega_kP_{\lambda_k}(x_k)
        -\omega_{k-1}D_f(y_{k-1},x_k)
        -\alpha_kD_f(u,x_k).
        \label{eq:convex-one-step-main}
    \end{equation}
\end{lemma}

Stopping checks ensure finitely many BPG solves per iteration (Appendix~\ref{app:convex-finite}).
The acceptance test
\begin{equation}
    D_\phi(z_{k-1},\widehat z)
    \leq
    \omega P_\lambda(x)
    +\omega_{k-1}D_f(y_{k-1},x)
    \label{eq:convex-certificate}
\end{equation}
uses only the current trial, not $u$, and makes the sum of the first three terms on the right of~\eqref{eq:convex-one-step-main} nonpositive.
Since $D_f(u,x_k)\geq0$ by~\cref{assumption:convex}, every accepted iteration satisfies $\mathcal E_k^{\mathrm C}(u) \leq\mathcal E_{k-1}^{\mathrm C}(u)$.
The coefficient equation gives the following weight-growth bound, converting this decrease into an objective bound; see Appendix~\ref{app:convex-posterior-proof} for the proof.

\begin{lemma}[Growth of the convex weights]
    \label[lemma]{lem:convex-weight-growth}
    Suppose that \cref{ass:standing,assumption:subproblem,assumption:convex} hold.
    For the accepted parameters of GA-BPGc, $\omega_N\geq\frac12\left(\sum_{k=1}^N\sqrt{\lambda_k/\kappa_k}\right)^2$.
\end{lemma}

Using~\cref{lem:convex-weight-growth}, we obtain the following convergence bound in terms of the accepted stepsizes $\lambda_k$ and geometry-acceleration parameters $\kappa_k$; its proof is given in Appendix~\ref{app:convex-posterior-proof}.
\begin{theorem}[Convergence bound for convex objectives]
    \label{thm:convex-posterior}
    Suppose that \cref{ass:standing,assumption:subproblem,assumption:convex} hold.
    After $N\geq1$ completed iterations of GA-BPGc, for every $u\in\mathcal U_\phi$,
    \begin{equation}
        \Phi(y_N)-\Phi(u)
        \leq
        \frac{
            2D_\phi(u,z_0)
        }{
            \left(
                \sum_{k=1}^N\sqrt{\lambda_k/\kappa_k}
            \right)^2
        }.
        \label{eq:convex-posterior}
    \end{equation}
    If a minimizer $x^\star\in \cl\Omega$ belongs to $\dom \phi$, this is an optimality-gap bound with $u=x^\star$.
    In particular, $\lambda_k/\kappa_k\geq c>0$ yields $\Phi(y_N)-\Phi^\star=\O(N^{-2})$.
\end{theorem}

The choice of auxiliary geometry $\phi$ is important for achieving an $\O(N^{-2})$ rate.
In practice, we should choose $\phi$ that satisfies $D_\phi(u,z_0)<\infty$ for an optimal solution $u$ and $\interior\dom\phi=\Omega$, so that the convergence bound is finite and the domain check $\widehat z\in\Omega$ can be omitted.
The following theorem gives sufficient conditions for $\O(N^{-2})$ convergence of GA-BPGc.
The proof is given in Appendix~\ref{app:convex-sublevel-results}, and representative geometry pairs are listed in Table~\ref{tab:convex-kernel-conditions}.
\begin{theorem}[Sufficient conditions for convex acceleration]
    \label{thm:convex-kernel-certificate}
    Suppose that \cref{ass:standing,assumption:subproblem,assumption:convex} hold, and that $\phi$ is Legendre.
    Let $u \in \Omega$ be a minimizer of $\Phi$ over $\cl\Omega$, and let $y_0 \in \Omega \cap \dom\Phi$.
    Assume the following conditions:
    \begin{enumerate}
        \renewcommand{\labelenumi}{(\roman{enumi})}
        \item $\interior\dom\phi = \Omega$.

        \item $D_\psi(u,\cdot)$ is weakly convex relative to $\phi$, \ie, $D_\psi(u,\cdot)+\beta\phi$ is convex on $\Omega$ for some $\beta\geq0$.

        \item If $\beta>0$, then $D_\phi(y_0,x_j)/D_\psi(y_0,x_j)\to0$ for every sequence $(x_j)\subset\{x\in\Omega:\Phi(x)\leq\Phi(y_0)\}$ with $x_j\to v\in\partial\Omega$ or $\|x_j\|\to\infty$.

        \item The effective conjugate domains $\dom\psi^*$ and $\dom\phi^*$ are open.

        \item For every compact $\mathcal{K}\subset\Omega$, there exists $0<C_{\mathcal{K}}<\infty$ such that
            \begin{equation}
                \langle\nabla\psi(x)-\nabla\psi(v),z-w\rangle
                \leq C_{\mathcal{K}}\sqrt{D_\psi(x,v)D_\phi(w,z)}
                \qquad(x,v,w,z\in \mathcal{K}).
            \end{equation}
            This condition holds, in particular, if $\psi,\phi\in C^2(\Omega)$ have positive definite Hessians.
    \end{enumerate}
    Then Algorithm~\ref{alg:convex-full} achieves the accelerated rate $\Phi(y_N)-\Phi^\star=\O(N^{-2})$.
\end{theorem}

\subsection{Relatively strongly convex objectives}\label{subsec:strong}
We next consider GA-BPGsc in the relatively strongly convex setting.
\begin{assumption}[Relative strong convexity]
    \label[assumption]{ass:strong}
    $f$ is relatively strongly convex with respect to $\psi$, \ie, there exists a constant $\mu > 0$ such that $D_f(u, v) \geq \mu D_{\psi}(u, v)$ holds for all $u \in \cl \Omega$ and $v \in \Omega$.
\end{assumption}
Take $u\in\mathcal U_\psi$, where $\mathcal U_\psi\coloneq \cl\Omega\cap\dom \Phi\cap\dom \psi$, and define the Lyapunov function as follows:
\begin{equation}
    \mathcal E_k^{\mathrm S}(u)
    \coloneq
    \omega_k(\Phi(y_k)-\Phi(u))+\theta_kD_\psi(u,z_k).
    \label{eq:strong-energy}
\end{equation}
We present GA-BPGsc in~\cref{alg:strong-full} in Appendix~\ref{app:strong} and explain its structure here.
\begin{itemize}
    \item \textbf{Coupling.} Let $\alpha>0$ satisfy $\kappa \alpha^2 = 2 \lambda \omega \theta$, where $\theta=\theta_{k-1} + \mu \alpha$ and $\kappa > 2 \mu \lambda$.
    The weights are initialized with $\omega_0 = 0$, and $\theta_0 = 1$.
    \item \textbf{Mirror update.} Set $\phi=\psi$ in GA-BPGc and replace its mirror update with
    \begin{equation*}
        \widehat z = \nabla \psi^{\ast} \left( \frac{\theta_{k-1} \nabla \psi(z_{k-1}) + \mu \alpha \nabla \psi(x) - \alpha g_\lambda(x)}{\theta} \right).
    \end{equation*}
    \item \textbf{Criterion.} GA-BPGsc accepts the trial if
    \begin{equation}
        \theta D_\psi(z_{k-1},\widehat z)
        -\mu\left(\omega_{k-1}D_\psi(y_{k-1},x)+\alpha D_\psi(z_{k-1},x)\right)
        \leq \omega P_\lambda(x). \label{eq:strong-certificate}
    \end{equation}
    By \cref{lem:strong-one-step}, this ensures that the Lyapunov function $\mathcal E_k^{\mathrm S}(u)$ does not increase.
    With the stopping checks, each iteration finishes after finitely many BPG solves (Appendix~\ref{app:strong-finite}).
\end{itemize}

The coupling and weighted mirror identities in Appendix~\ref{app:common-identities}, together with Lemma~\ref{lem:bpg-comparison} and relative strong convexity, yield the following estimate.
See Appendix~\ref{app:strong-analysis} for the proof.

\begin{lemma}[Relatively strongly convex one-step estimate]
    \label[lemma]{lem:strong-one-step}
    Suppose that \cref{ass:standing,assumption:subproblem,ass:strong} hold.
    For every accepted iteration $k\geq1$ of GA-BPGsc and every $u\in\mathcal U_\psi$,
    \begin{align}
        \mathcal E_k^{\mathrm S}(u)
        -\mathcal E_{k-1}^{\mathrm S}(u)
        &\leq \theta_kD_\psi(z_{k-1},z_k)
        -\omega_kP_{\lambda_k}(x_k)
        \notag\\
        &\quad-\mu\left(\omega_{k-1}D_\psi(y_{k-1},x_k)+\alpha_kD_\psi(z_{k-1},x_k)\right).
        \label{eq:strong-one-step}
    \end{align}
\end{lemma}

The criterion~\eqref{eq:strong-certificate} uses only the current trial, not $u$, and ensures $\mathcal E_k^{\mathrm S}(u)\leq\mathcal E_{k-1}^{\mathrm S}(u)$.
The coefficient equation links this decrease to an objective bound.
For accepted parameters, $\kappa_k>2\mu\lambda_k$ ensures $\sqrt{2\mu\lambda_k/\kappa_k}\in(0,1)$.
The weights satisfy
\begin{equation}
    \omega_N\geq \omega_1\prod_{k=2}^N
    \left(1-\sqrt{\frac{2\mu\lambda_k}{\kappa_k}}\right)^{-1},
    \qquad N\geq1,
    \label{eq:strong-A-growth}
\end{equation}
where an empty product equals one.
Using this growth estimate, we obtain the following convergence bound in terms of the accepted ratios $\lambda_k/\kappa_k$; the proofs of the growth estimate and the theorem are given in Appendix~\ref{app:strong-analysis}.

\begin{theorem}[Convergence bound for relatively strongly convex objectives]
    \label{thm:strong-main}
    Suppose that \cref{ass:standing,assumption:subproblem,ass:strong} hold.
    After $N\geq1$ completed iterations of GA-BPGsc, for every $u\in\mathcal U_\psi$,
    \begin{equation}
        \Phi(y_N)-\Phi(u)
        \leq
        \frac{D_\psi(u,z_0)}{\omega_1}
        \prod_{k=2}^N\left(1-\sqrt{\frac{2\mu\lambda_k}{\kappa_k}}\right).
        \label{eq:strong-product}
    \end{equation}
    If a minimizer $x^\star\in\cl\Omega$ belongs to $\dom\psi$, this is an optimality-gap bound with $u=x^\star$.
    In particular, a uniform bound $\sqrt{2\mu\lambda_k/\kappa_k}\geq q_{\min}>0$ yields linear convergence of the objective gap.
\end{theorem}

The following theorem gives sufficient conditions for linear convergence of GA-BPGsc.
The proof and applications to Euclidean, KL, and Burg geometries are given in Appendix~\ref{app:sc-acceleration-conditions}.

\begin{theorem}[Sufficient conditions for strong acceleration]
    \label{thm:sc-acceleration-conditions}
    Suppose that \cref{ass:standing,assumption:subproblem,ass:strong} hold, and let $u\in\Omega$ be a minimizer of $\Phi$ over $\cl\Omega$.
    Assume conditions~(i)--(v) of Theorem~\ref{thm:convex-kernel-certificate} with $\phi=\psi$.
    Then GA-BPGsc satisfies $\Phi(y_N) - \Phi^{\ast} = \O(q^N)$ for some $q\in(0,1)$.
\end{theorem}

\subsection{Nonconvex objectives}\label{subsec:nonconvex}
For $x\in\Omega\cap\dom\rho$ and $\widehat y\in T_\lambda(x)$, define
\begin{equation*}
    \mathcal R_\lambda(x)
    \coloneq
    -\frac{2}{\lambda}\left(\langle\nabla f(x),\widehat y-x\rangle+\rho(\widehat y)-\rho(x)
    +\lambda^{-1}D_\psi(\widehat y,x)\right).
\end{equation*}
This normalized Bregman forward--backward residual vanishes exactly at interior stationary points and satisfies $\mathcal R_\lambda(x)\geq 2D_\psi(x,\widehat y)/\lambda^2$; see Appendix~\ref{app:common-gap}.
The envelope interpretation is discussed in \citet{ahookhosh2021bregman,fatkhullin2024taming}.
We present GA-BPGnc in~\cref{alg:nonconvex-full} in Appendix~\ref{app:nonconvex} and explain its structure here.
\begin{itemize}
    \item \textbf{Coupling \& mirror update.} GA-BPGnc couples $y_k$ and $y_{k-1}$ by extrapolation in the mirror space induced by a fixed Legendre function $\phi$, with an arbitrary weight $\beta_{k+1}\geq0$:
    \begin{equation*}
        \widehat z = \nabla \phi^{\ast}(\nabla \phi(y_k) + \beta_{k+1}(\nabla\phi(y_k) - \nabla \phi(y_{k-1}))).
    \end{equation*}
    We set $z_k=\widehat z$ if extrapolation is permitted and $\widehat z$ is well defined with $\widehat z\in\Omega\cap\dom\rho$; otherwise, we set $z_k=y_k$ and start the next BPG update from $y_k$.
    \item \textbf{Criterion.} GA-BPGnc sets $y_k=\widehat y$ if $\Phi(\widehat y)<\Phi(y_{k-1})$, and $y_k=y_{k-1}$ otherwise.
    GA-BPGnc rejects the extrapolation if $\delta_k < \frac{\sigma\lambda}{2}\mathcal R_\lambda(x)$, where $\delta_k \coloneq \Phi(y_{k-1}) - \Phi(\widehat{y})$ and $\sigma \in (0,1)$ is the reset parameter.
\end{itemize}

Take a nonterminal iteration, for which $\mathcal R_{\lambda_k}(x_k)>0$, and define $[t]_+\coloneq\max(t,0)$ for $t\in\R$ and
\begin{equation*}
    \tau_k\coloneq\frac{2\delta_k}{\lambda_k\mathcal R_{\lambda_k}(x_k)}.
\end{equation*}
The acceptance rule relates the objective decrease to $\frac{\lambda_k}{2}\mathcal R_{\lambda_k}(x_k)$; the reset rule then ensures $\tau_{k+1} \geq 1$ whenever $\tau_k < \sigma$ and the next iteration is nonterminal.
Its proof is given in~\cref{app:nonconvex-descent}.

\begin{lemma}[Descent and a safe next iteration]
    \label[lemma]{lem:nonconvex-descent}
    Suppose that \cref{ass:standing,assumption:subproblem} hold.
    Every nonterminal iteration of GA-BPGnc satisfies
    \begin{equation}
        \Phi(y_{k-1})-\Phi(y_k)
        =
        [\delta_k]_+
        =
        \frac{\lambda_k}{2}[\tau_k]_+\mathcal R_{\lambda_k}(x_k).
        \label{eq:nonconvex-descent}
    \end{equation}
    If $x_k=y_{k-1}$, then $\tau_k\geq1$; if $\tau_k<\sigma$, then $x_{k+1}=y_k$; hence $\tau_{k+1}\geq1$ if step $k+1$ is nonterminal.
\end{lemma}

Since $\frac{\lambda_k}{2}[\tau_k]_+\mathcal R_{\lambda_k}(x_k)\geq0$, the acceptance rule ensures $\Phi(y_k) \leq \Phi(y_{k-1})$.
Summing these decreases gives a stationarity bound in terms of $\sum_{k\in\mathcal A_N}\tau_k\lambda_k$, with $\mathcal A_N$ defined below.
The reset rule and stepsize lower bound control this sum from below, yielding an $\O(N^{-1})$ worst-case bound.
The following theorem summarizes these bounds; see Appendices~\ref{app:nonconvex-posterior},~\ref{app:nonconvex-worstcase}, and~\ref{app:nonconvex-cost} for the proof.

\begin{theorem}[Stationarity bounds for nonconvex objectives]
    \label{thm:nonconvex-main}
    Suppose that \cref{ass:standing,assumption:subproblem} hold and GA-BPGnc has not returned a stationary point in its first $N \geq 2$ iterations.
    Let $\mathcal A_N\coloneq\{1\leq k\leq N:\delta_k>0\}$ and $\underline\lambda\coloneq \min\{\lambda_0,1/(\gamma_-L)\}>0$.
    Then
    \begin{equation}
        \min_{k\in\mathcal A_N}
        \mathcal R_{\lambda_k}(x_k)
        \leq
        \frac{2(\Phi(y_0)-\Phi^\star)}{\sum_{k\in\mathcal A_N}\tau_k\lambda_k}
        \leq
        \frac{4(\Phi(y_0)-\Phi^\star)}{\sigma\underline\lambda(N-1)}.
        \label{eq:nonconvex-posterior}
    \end{equation}
    Thus a residual at most $\varepsilon^2$ is obtained within $\O(\varepsilon^{-2})$ outer iterations and BPG subproblem solves.
\end{theorem}
The first bound in~\eqref{eq:nonconvex-posterior} depends on the values of $\tau_k$ and $\lambda_k$ obtained during the iterations.
The second follows from the reset rule and the stepsize lower bound; it depends only on $N$, the initial objective gap, $\sigma$, and $\underline\lambda$.
Unlike Sections~3.1 and~3.2, the proof sums objective decreases rather than using comparison-point Lyapunov bounds.
Under a uniform subgradient comparison, the residual bound implies a norm-based stationarity bound (Appendix~\ref{app:nonconvex-stationarity}).

\section{Numerical experiments}
\label{sec:experiments}

We compare GA-BPG variants with Bregman baselines using objective gaps versus iterations and wall-clock time.
For an objective $\Phi$ and reference value $\Phi_{\mathrm{ref}}$, we plot the normalized objective gap $(\Phi(x_k)-\Phi_{\mathrm{ref}})/(\Phi(x_0)-\Phi_{\mathrm{ref}})$, where $x_k$ is the reported iterate.
The reference is a known optimal value, an empirical objective value, or a certified lower bound, as indicated below.
GA-BPGc and GA-BPGsc use $\gamma_+=1/0.9$, $\gamma_-=2$, and $\gamma_\kappa=1.5$.
GA-BPGnc uses $\gamma_+=1.1$, $\gamma_-=2$, $\sigma=0.5$, and $\lambda_{\max}=10^6\lambda_0$.
For ABPG-g in the convex and relatively strongly convex experiments, we use the Python \texttt{accbpg} implementation\footnote{\url{https://github.com/microsoft/accbpg}} with $\gamma=2$, its linesearch settings, and problem-specific relative-smoothness constants.
We choose the initial gain to match the first trial proximal stepsize of the corresponding GA-BPG method.
Additional experiments and more discussions are given in Appendices~\ref{app:additional-experiments} and~\ref{app:discussions}, respectively.
All experiments were run in Python 3.12.3 on an Intel Core i7-13620H under Ubuntu 24.04 (WSL2).

\paragraph{Convex Poisson.}
Given $A\in\R^{m\times n}$ and $b\in\R^m$, we consider the convex Poisson inverse problem $\min_{x \in \R^n}\Phi_0(x)\coloneq \frac{1}{m}D_{\mathrm{KL}}(b,Ax) + \delta_{[0,10^3]^n}(x)$, where $D_{\mathrm{KL}}(u,v) = \sum_i(u_i\log\frac{u_i}{v_i}-u_i+v_i)$ is the Kullback--Leibler divergence.
We use the NYTimes bag-of-words data \citep{newman2008bagofwords}, which originally contain $300{,}000$ documents and $102{,}660$ terms.
After removing empty rows and columns and scaling each retained row to sum to $100$, we obtain a sparse matrix $A$ with $m=299{,}752$ and $n=101{,}636$.
For each seed, we generate a positive planted vector $x^\star\in(0,10^3)^n$ and set $b=Ax^\star$, \ie, $\Phi_0^\star=\Phi_0(x^\star)=0$, which we use as the reference value.
All methods use the Burg geometry $\psi(x)\coloneq-\sum_j\log x_j$ for~\eqref{eq:bpg-subproblem}.
GA-BPGc additionally uses the negative-entropy auxiliary geometry $\phi(x)=\sum_j(x_j\log x_j-x_j)$ for its mirror update.
All methods start from $y_0=\mathbf{1}$.
Runs stop upon meeting algorithm-specific stopping criteria or reaching $5{,}000$ iterations or $1{,}800$ seconds.
Figure~\ref{fig:convex-nytimes} compares GA-BPGc with BPG using linesearch (BPG-LS) and ABPG-g \citep{hanzely2021accelerated},  with dotted bounds from~\eqref{eq:convex-telescope}.
Our algorithm outperforms the baselines.
To analyze how stepsize adaptation and the choice of auxiliary geometry each contribute to this fast convergence, we conduct additional experiments in Appendix~\ref{app:additional-experiments} (Figure~\ref{fig:additional-objective-gap}).

\begin{figure}[!tp]
    \centering
    \includegraphics[width=0.9\linewidth]{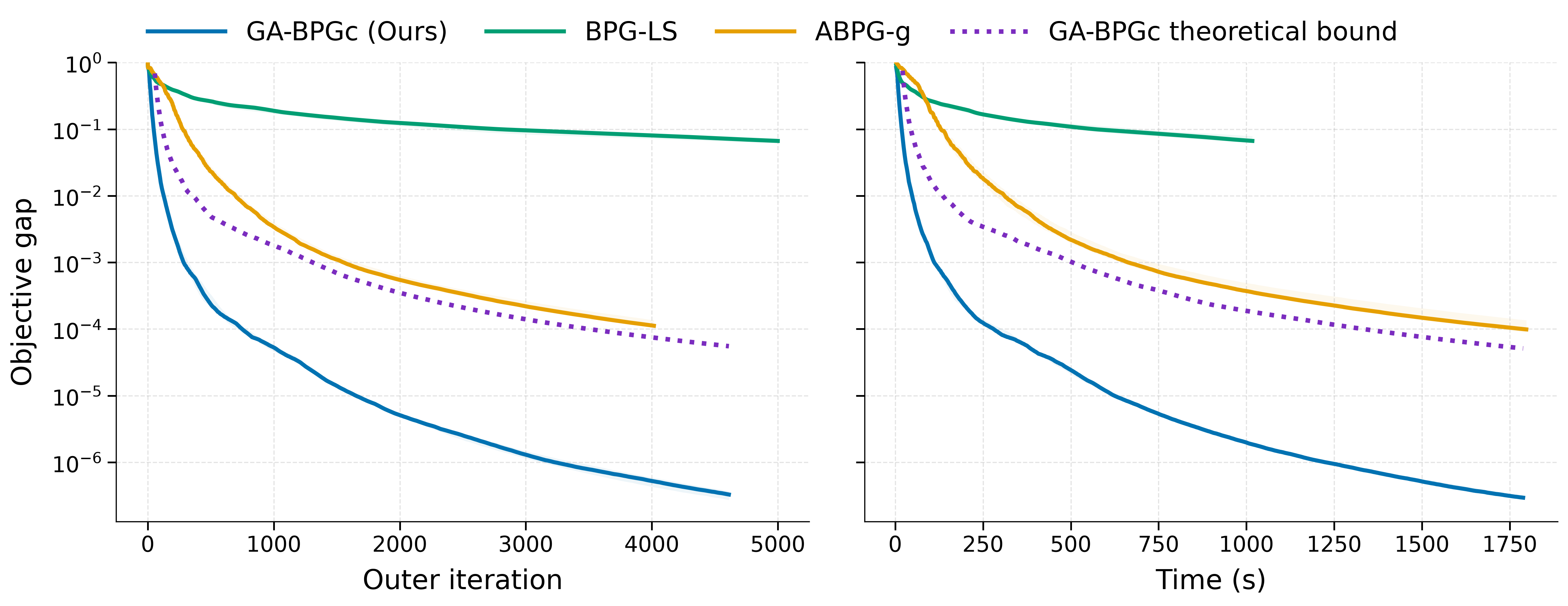}
    \caption{Log plot of the normalized objective gap for convex Poisson on the NYTimes data.}
    \label{fig:convex-nytimes}
\end{figure}

\paragraph{Relatively strongly convex objectives.}
Motivated by stacked regression \citep{breiman1996stacked} and entropy-penalized convex aggregation \citep{koltchinskii2009sparse}, we consider least-squares aggregation with a Kullback--Leibler penalty toward uniform weights: $\min_{w\in\Delta_d}F_\mu(w)\coloneqq\frac{1}{2m}\|Hw-b\|^2+\mu D_{\mathrm{KL}}(w,\pi)$, where $\pi=\frac{1}{d}\mathbf{1}$ and $\Delta_d = \left\{ w\in\R^d: w\geq0,\ \mathbf{1}^\top w=1 \right\}$, $b\in\R^m$ is the response vector, and $H\in\R^{m\times d}$ is a dense out-of-fold prediction matrix constructed from simulated Gaussian regression data using 10-fold cross-validation.
The columns of $H$ comprise $1600$ backward-selected least-squares predictors and $1600$ ridge predictors, giving $m=2400$ and $d=3200$.
Working on the affine hull of $\Delta_d$, all algorithms use $f=F_\mu$, $\rho=0$, and $\psi(w)=\sum_j w_j\log w_j$, with $f,\psi=+\infty$ outside $\Delta_d$.
The objective is $\mu$-strongly convex relative to $\psi$; GA-BPGsc uses $\phi=\psi$ with the normalized exponential inverse mirror map.
The minimizer lies in $\operatorname{ri}\Delta_d$, so Theorem~\ref{thm:sc-acceleration-conditions} guarantees bounded $\kappa_k$ and linear convergence for this experiment.
All methods use $w_0=\pi$ and $\mu=10^{-3}$.
Runs stop upon meeting algorithm-specific criteria or reaching $10{,}000$ iterations or $600$ seconds.
Figure~\ref{fig:strong-aggregation} compares GA-BPGsc with BPG-LS and ABPG-g \citep{hanzely2021accelerated}, with dotted bounds from~\eqref{eq:strong-telescope}.
For each seed, we define $F_{\mathrm{ref}}$ as the smallest objective value over all recorded primary iterates of the three methods.
GA-BPGsc reaches high-accuracy solutions fastest in wall-clock time.

\begin{figure}[!tp]
    \centering
    \includegraphics[width=0.9\linewidth]{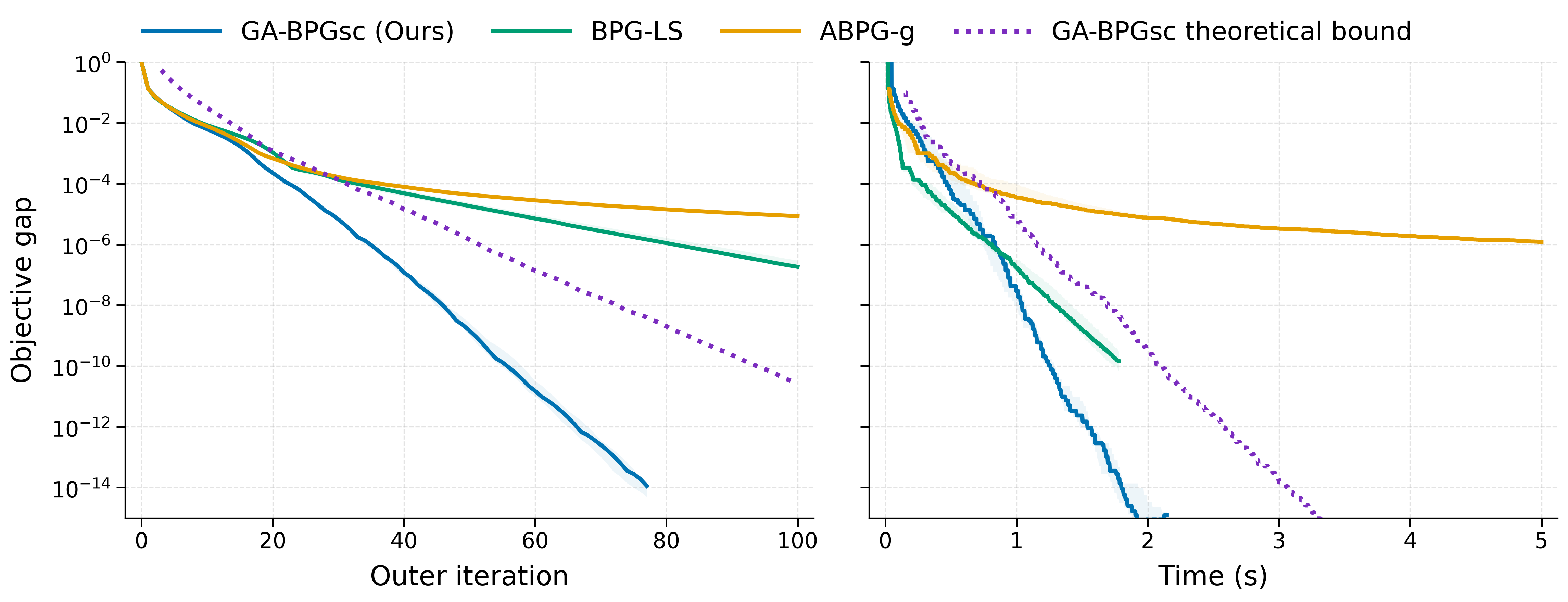}
    \caption{Log plot of the normalized objective gap for entropy-regularized least squares with $\mu=10^{-3}$.}
    \label{fig:strong-aggregation}
\end{figure}

\paragraph{Nonnegative matrix factorization.}
We factorize the MovieLens 100K ratings by solving $\min_{W\geq0,\,H\geq0}F(W,H)\coloneqq \frac{1}{2|\O|}\sum_{(i,j)\in\O} (W_{i:}H_{:j}-r_{ij})^2$, where $\O$ contains the $100{,}000$ observed ratings, $W\in\R^{943\times16}$, and $H\in\R^{16\times1682}$.
Missing ratings are excluded from the objective.
We compare GA-BPGnc with BPG \citep{bolte2018first}, BPGe \citep{zhang2019bregman}, and CoCaIn-BPG \citep{mukkamala2020convex}.
All methods use $\psi(x)=\frac14(\|W\|_{\F}^2+\|H\|_{\F}^2)^2+\frac12(\|W\|_{\F}^2+ \|H\|_{\F}^2) $; GA-BPGnc also uses $\phi=\psi$.
BPG and BPGe use the step $1/L$, and GA-BPGnc and CoCaIn-BPG use their respective adaptive rules.
The five seeds specify common initializers for all methods.
Each method runs for at most $100$ seconds.
For each seed, $F_{\mathrm{ref}}$ is the smallest objective recorded across all four methods during these runs.
Figure~\ref{fig:nmf} shows the median and interquartile range of the normalized empirical objective gap over the first $90$ seconds.
GA-BPGnc reaches smaller empirical gaps than the baselines in both outer iterations and wall-clock time.

\begin{figure}[!tp]
    \centering
    \includegraphics[width=0.9\linewidth]{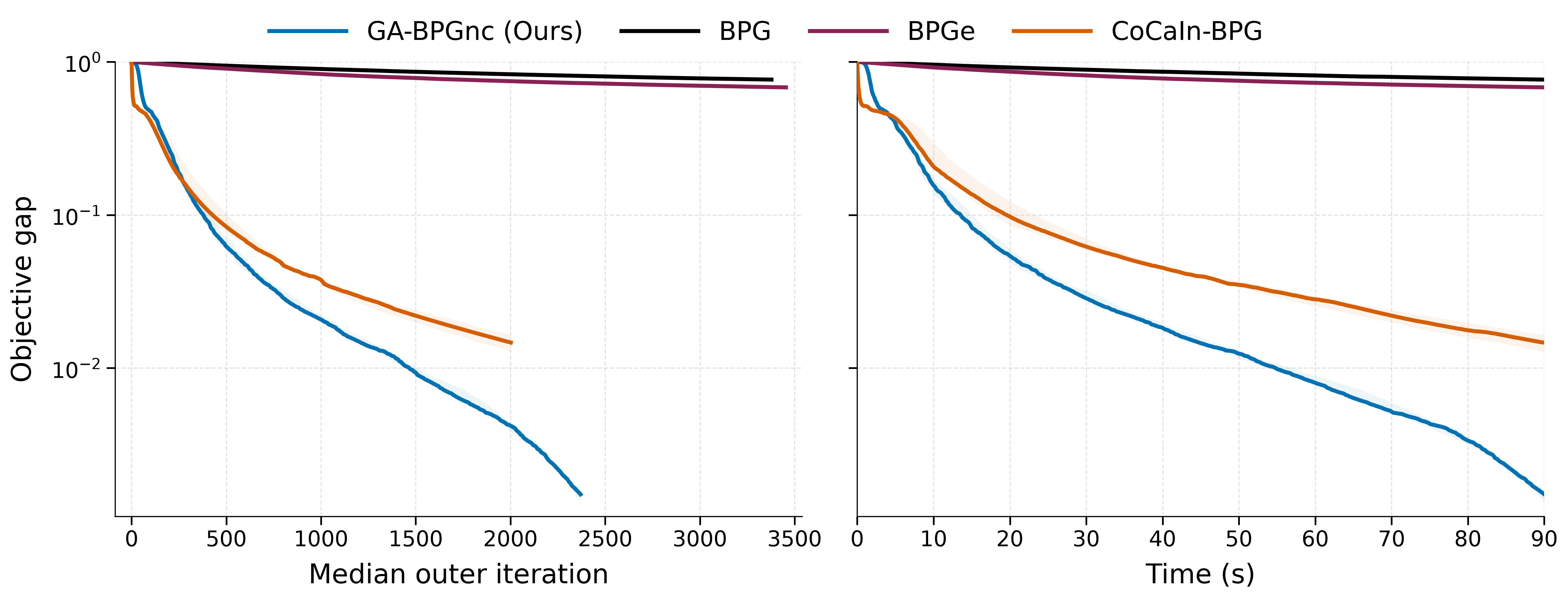}
    \caption{Log plot of the normalized empirical objective gap for nonnegative matrix factorization on MovieLens 100K.}
    \label{fig:nmf}
\end{figure}

\section{Conclusion}
\label{sec:conclusion}

We proposed GA-BPG, a family of BPG algorithms with auxiliary mirror updates, for convex, relatively strongly convex, and nonconvex objectives.
Auxiliary mirror points guide BPG steps, with computable tests for acceptance, parameter updates, and resets.
Under standard BPG assumptions, GA-BPGc and GA-BPGsc attain $\O(k^{-2})$ and linear convergence, respectively, when $\sup_k\kappa_k<\infty$.
For GA-BPGnc, we established an $\O(k^{-1})$ bound on a stationarity residual without assuming $\Omega=\R^n$.
The algorithms require no triangle-scaling condition or prior knowledge of the global relative-smoothness constant.
Experiments show faster convergence than the Bregman baselines in all three regimes.

We established general sufficient conditions on Bregman divergences used in our algorithms that ensure $\O(k^{-2})$ convergence of GA-BPGc, and linear convergence of GA-BPGsc.
These conditions cover Burg--KL and KL--KL geometries as representative examples.
A sharper characterization of geometry pairs that yield bounded $\kappa_k$, and methods to prevent rate deterioration when $\kappa_k$ grows, remain topics for future work.

\subsection*{AI use statement}
In this work, we used generative AI tools for formulating mathematical claims, providing key proof ingredients, assisting with proof writing, proposing or refining hypotheses, providing feedback on research methodology and experiments, implementing algorithms, assisting with translation, developing theoretical models or conceptual frameworks, and cleaning and reformatting datasets.
We have not used generative AI tools for directly generating synthetic datasets, or qualitative and thematic data analysis.
Additionally, we used generative AI tools for creating or modifying scientific figures, writing or editing code, drafting manuscript text, reviewing our discussion of related work, searching for information, editing for readability, identifying relevant literature.
We have reviewed all AI-assisted work.
The code was checked by one author, and the proofs were verified by the authors.
We take responsibility for the final content of this work, including text, claims, or artifacts produced with the aid of generative AI.

\subsection*{Ethics statement}
This work concerns the development and theoretical analysis of optimization algorithms, with numerical experiments on synthetic and existing benchmark datasets.
We do not identify any specific ethical concerns arising from this work.

\subsection*{Reproducibility statement}
The algorithms and assumptions are described in \cref{sec:setting-methods,sec:convergence}, with proofs provided in \cref{app:common-results,app:convex,app:strong,app:nonconvex}.
Detailed pseudocode is given in \cref{alg:convex-full,alg:strong-full,alg:nonconvex-full}.
The experimental settings and datasets are described in \cref{sec:experiments}.
The source code for the numerical experiments is included in the supplementary material.

\section*{Acknowledgements}
This project has been funded by the Japan Society for the Promotion of Science (JSPS); JSPS KAKENHI Grant Number JP23K28041 and JP25K21156.

\bibliographystyle{plainnat}
\bibliography{main}

\appendix

\section{Additional related work}
\label{app:related-work}

\paragraph{Relative smoothness and BPG.}
\citet{bauschke2017descent} replace a Euclidean quadratic upper bound by a Bregman upper bound, allowing proximal-gradient updates beyond globally Lipschitz gradients.
\citet{lu2018relatively} develop relative smoothness and relative strong convexity and analyze first-order algorithms under these conditions.
\citet{Ou2026-wy} develop adaptive BPG algorithms for convex composite problems under local relative smoothness, selecting stepsizes from local curvature estimates without backtracking linesearch.
\citet{Malitsky2025-np} propose Polyak-type stepsizes for entropic mirror descent on consistent linear systems, establishing convergence guarantees and quantifying the implicit bias toward solutions with small $\ell_1$ norm when initialized near the origin.
\citet{Kunstner2026-fe} analyze mirror Polyak methods based on Bregman projections, with rates adapting to relative smoothness and relative strong convexity, and use a primal--dual lifting to avoid requiring the optimal value for structured convex problems.
\citet{De-Marchi2026-yo} establish $\O(\log k/k)$ objective convergence rates for mirror descent with logarithmic barriers, including cases with boundary minimizers where standard Bregman-distance bounds are infinite.
Beyond proximal methods, \citet{Takahashi2026-ch} develop Frank--Wolfe algorithms with adaptive Bregman stepsizes for relatively smooth convex and weakly convex objectives.

\paragraph{Acceleration and its limitations.}
The Euclidean accelerated proximal-gradient rate \citep{nesterov1983method,beck2009fast} does not extend uniformly to arbitrary relatively smooth problems.
\citet{dragomir2022optimal} establish a matching $\O(1/k)$ upper and lower complexity order for the Bregman first-order oracle class considered there.
\citet{Kim2023-yd} establish mirror duality between mirror-descent-type methods for reducing objective values and gradient magnitudes, and derive a dual accelerated mirror descent method for convex objectives with Lipschitz gradients and strongly convex distance-generating functions.
\citet{hanzely2021accelerated} obtain accelerated rates under triangle-scaling conditions and develop algorithms that adjust gains during the iterations and bound the objective gap in terms of the accepted gains.
GA-BPG instead tests the auxiliary mirror displacement through its acceptance inequality.
Our bounds yield $\O(k^{-2})$ and linear rates when the accepted ratios $\lambda_k/\kappa_k$ are bounded away from zero.

\paragraph{Relative strong convexity and linear convergence.}
Relative strong convexity supplies a lower bound in the same Bregman geometry as the upper relative-smoothness bound \citep{lu2018relatively}.
Linear convergence beyond the usual Euclidean strong-convexity and Lipschitz-gradient assumptions is also studied by \citet{bauschke2019linear}, using generalized gradient-domination conditions.
Our analysis of GA-BPGsc assumes relative strong convexity and uses the known constant $\mu$ in both the weighted mirror update and the acceptance inequality.
Its product bound yields a uniform linear rate when the accepted ratios $\lambda_k/\kappa_k$ are bounded away from zero.

\paragraph{Nonconvex Bregman algorithms.}
\citet{bolte2018first} analyze nonconvex BPG algorithms using Bregman upper bounds and smooth-adaptability conditions.
BPGe combines BPG with extrapolation \citep{zhang2019bregman}, whereas CoCaIn-BPG controls extrapolation and stepsizes through upper and lower curvature estimates \citep{mukkamala2020convex}.
For DC objectives, BPDCA and BPDCAe use the convex components of the objective explicitly \citep{takahashi2022new}.
Approximate BPG simplifies subproblems using a local quadratic approximation of the Bregman divergence \citep{takahashi2025approximate}, and \citet{Fujiki2025-do} extend this approach with a variable metric linesearch.
\citet{Ding2026-qe} construct relatively smooth nonconvex examples in which bounded entropic mirror-descent sequences have non-KKT boundary accumulation points despite nonincreasing objective values.
For GA-BPGnc, we establish a finite-time stationarity bound under the standing assumptions.

\paragraph{Forward--backward envelopes.}
Bregman forward--backward envelopes provide merit functions and a way to assess candidate directions beyond the underlying BPG update \citep{ahookhosh2021bregman}.
Related normalized envelope gaps are used as stationarity measures in nonconvex mirror-descent analysis \citep{fatkhullin2024taming}.
We use the gap between the objective and the Bregman forward--backward envelope to define $\mathcal R_\lambda$.
An additional uniform comparison bound relates this residual to subgradient norms, as stated in Appendix~\ref{app:nonconvex-stationarity}.

\section{Common BPG properties}
\label{app:common-results}

Throughout, Assumptions~\ref{ass:standing} and~\ref{assumption:subproblem} hold.
BPG base points and solutions belong to $\Omega$.
We use the extended first-argument conventions stated in Section~1; no gradient is evaluated at a boundary comparison point.

\subsection{The BPG map}
\label{app:common-bpg-map}

\begin{proposition}[Uniqueness and optimality]
    \label[proposition]{prop:bpg-unique}
    For every admissible $(x,\lambda)$, the minimizer $y\in T_\lambda(x)$ over $\cl\Omega$ is unique, belongs to $\Omega\cap\dom \rho$, and satisfies
    \begin{equation}
        g_\lambda(x)-\nabla f(x)\in\partial\rho(y).
        \label{eq:bpg-optimality}
    \end{equation}
    For $x\in\Omega\cap\dom \rho$,
    \begin{equation}
        x\in T_\lambda(x)
        \quad\Longleftrightarrow\quad
        0\in\nabla f(x)+\partial\rho(x).
        \label{eq:bpg-fixed-point}
    \end{equation}
\end{proposition}

\begin{proof}
    Fix an admissible pair $(x,\lambda)$.
    Existence of a minimizer follows from Assumption~\ref{assumption:subproblem}.
    Set
    \begin{equation*}
        H=\psi+\lambda\rho,
        \qquad
        p=\nabla\psi(x)-\lambda\nabla f(x).
    \end{equation*}
    Since $\dom\psi\subset\cl\Omega$, the BPG subproblem is equivalent to minimizing $H(u)-\langle p,u\rangle$ over $\R^n$.

    Let $y$ be any minimizer.
    The function $\psi$ is continuous at every point of $\Omega$, and $\Omega\cap\dom\rho\neq\emptyset$.
    Thus the subdifferential sum rule applies, and optimality gives
    \begin{equation*}
        p\in\partial H(y)
        =\partial\psi(y)+\lambda\partial\rho(y).
    \end{equation*}
    In particular, $\partial\psi(y)\neq\emptyset$.
    Essential smoothness of $\psi$ gives $\dom\partial\psi=\interior\dom\psi=\Omega$, so $y\in\Omega\cap\dom\rho$.
    Since $\partial\psi(y)=\{\nabla\psi(y)\}$, the preceding inclusion becomes
    \begin{equation*}
        \frac{\nabla\psi(x)-\nabla\psi(y)}{\lambda}
        -\nabla f(x)
        \in\partial\rho(y),
    \end{equation*}
    which proves~\eqref{eq:bpg-optimality}.

    Every minimizer is therefore interior.
    Strict convexity of $\psi$ on $\Omega$ and convexity of $\rho$ make $H(u)-\langle p,u\rangle$ strictly convex on $\Omega\cap\dom\rho$, so the minimizer is unique.

    For $x\in\Omega\cap\dom\rho$, the inclusion $x\in T_\lambda(x)$ implies stationarity by~\eqref{eq:bpg-optimality}.
    Conversely, if $0\in\nabla f(x)+\partial\rho(x)$, then
    \begin{equation*}
        p\in\nabla\psi(x)+\lambda\partial\rho(x)=\partial H(x).
    \end{equation*}
    Hence $x$ minimizes the BPG subproblem, so $x\in T_\lambda(x)$.
    This proves~\eqref{eq:bpg-fixed-point}.
\end{proof}

\begin{proposition}[Continuity for a fixed stepsize]
    \label{prop:bpg-continuity}
    Fix $\lambda>0$.
    If $x_j\to x\in\Omega$ and the BPG subproblems at $x_j$ and $x$ have interior solutions $y_j\in T_\lambda(x_j)$ and $y\in T_\lambda(x)$, respectively, then $y_j\to y$.
\end{proposition}

\begin{proof}
    Put $H=\psi+\lambda\rho$ and $p(x)=\nabla\psi(x)-\lambda\nabla f(x)$.
    The function $H$ is proper, lower semicontinuous, and convex, and the BPG subproblem is equivalent to minimizing $H(u)-\langle p(x),u\rangle$ over $\R^n$.
    Its unique solution implies
    \begin{equation*}
        \partial H^*(p(x))=T_\lambda(x)=\{y\}.
    \end{equation*}
    Moreover, $p(x)\in\interior\dom H^*$: otherwise a nonzero supporting normal to the convex domain at $p(x)$ could be added to any subgradient, contradicting this singleton identity.

    The subdifferential of a finite-dimensional closed convex function is locally bounded in the interior of its domain.
    Since $p(x_j)\to p(x)$ and $y_j\in\partial H^*(p(x_j))$, these solutions are bounded.
    Closedness of the subdifferential graph places every cluster point in $\partial H^*(p(x))$.
    Uniqueness then gives the claimed convergence.
\end{proof}

\subsection{Comparison and stopping tests}
\label{app:common-comparison}

\begin{proof}[Proof of Lemma~\ref{lem:bpg-comparison}]
    Let $y\in T_\lambda(x)$ and $u\in \cl\Omega\cap\dom \Phi$.
    By~\eqref{eq:bpg-optimality} and convexity of $\rho$,
    \begin{equation*}
        \rho(y)-\rho(u)
        \leq\langle g_\lambda(x)-\nabla f(x),y-u\rangle.
    \end{equation*}
    Also,
    \begin{equation*}
        f(y)-f(u)
        =
        D_f(y,x)-D_f(u,x)+\langle\nabla f(x),y-u\rangle.
    \end{equation*}
    Adding these expressions and using $\langle g_\lambda(x),x-y\rangle=[D_\psi(x,y)+D_\psi(y,x)]/\lambda$ gives~\eqref{eq:bpg-comparison}.

    Under~\eqref{eq:local-rs-test},
    \begin{equation*}
        P_\lambda(x)
        =
        \frac{D_\psi(x,y)+D_\psi(y,x)}{\lambda}-D_f(y,x)
        \geq
        \frac{D_\psi(x,y)}{\lambda}.
    \end{equation*}
    Strict convexity of $\psi$ on $\Omega$ gives $P_\lambda(x)=0\Rightarrow x=y$; the converse follows by substitution.
    The common point is stationary by~\eqref{eq:bpg-fixed-point}.
\end{proof}

\begin{corollary}[Relative lower bound]
    \label{cor:bpg-relative-lower}
    Let $y\in T_\lambda(x)$.
    If $D_f(u,x)\geq\mu D_\psi(u,x)$ for a comparison point $u\in \cl\Omega\cap\dom \Phi\cap\dom \psi$, then
    \begin{equation}
        \Phi(y)-\Phi(u)
        \leq
        \langle g_\lambda(x),x-u\rangle-P_\lambda(x)
        -\mu D_\psi(u,x).
        \label{eq:bpg-relative-lower}
    \end{equation}
\end{corollary}

\begin{proof}
    Substitute the lower bound in~\eqref{eq:bpg-comparison}.
\end{proof}

\subsection{The residual and stationarity}
\label{app:common-gap}

\begin{proposition}[Residual, descent, and stationarity]
    \label{prop:bpg-model-gap}
    For $x\in\Omega\cap\dom \rho$ and $y\in T_\lambda(x)$,
    \begin{equation}
        \mathcal R_\lambda(x)\geq\frac{2}{\lambda^2}D_\psi(x,y)\geq0,
        \qquad
        \mathcal R_\lambda(x)=0
        \quad\Longleftrightarrow\quad
        y=x.
        \label{eq:gap-lower}
    \end{equation}
    If~\eqref{eq:local-rs-test} holds, then
    \begin{equation}
        \Phi(y)\leq\Phi(x)-\frac{\lambda}{2}\mathcal R_\lambda(x).
        \label{eq:gap-descent}
    \end{equation}
    This residual vanishes exactly when $0\in\nabla f(x)+\partial\rho(x)$.
\end{proposition}

\begin{proof}
    Optimality and convexity of $\rho$ give
    \begin{equation*}
        \rho(x)-\rho(y)
        \geq
        \langle g_\lambda(x)-\nabla f(x),x-y\rangle.
    \end{equation*}
    Consequently,
    \begin{equation*}
        \mathcal R_\lambda(x)
        \geq
        \frac{2}{\lambda}\left(
        \langle g_\lambda(x),x-y\rangle
        -\lambda^{-1}D_\psi(y,x)\right)
        =
        \frac{2}{\lambda^2}D_\psi(x,y).
    \end{equation*}
    Strict convexity and~\eqref{eq:bpg-fixed-point} give the equivalences.
    The identity
    \begin{equation*}
        \Phi(y)-\Phi(x)
        =
        -\frac{\lambda}{2}\mathcal R_\lambda(x)-\lambda^{-1}D_\psi(y,x)+D_f(y,x)
    \end{equation*}
    proves~\eqref{eq:gap-descent} under the upper-bound test.
\end{proof}

The Bregman forward--backward envelope associated with this subproblem is
\begin{equation*}
    \mathcal F_\lambda(x)
    =
    f(x)+\min_{u\in \cl\Omega}
    \left\{
        \rho(u)+\langle\nabla f(x),u-x\rangle
        +\lambda^{-1}D_\psi(u,x)
    \right\}.
\end{equation*}
For $x\in\Omega\cap\dom \rho$, $\mathcal R_\lambda(x)=\frac{2}{\lambda}(\Phi(x)-\mathcal F_\lambda(x))$.
These identities require no differentiability assertion about $\mathcal F_\lambda$.

\begin{lemma}[A subgradient at the BPG point]
    \label{lem:bpg-subgradient}
    For $y\in T_\lambda(x)$,
    \begin{equation*}
        s_\lambda(x)
        \coloneq
        g_\lambda(x)+\nabla f(y)-\nabla f(x)
        \in\nabla f(y)+\partial\rho(y).
    \end{equation*}
\end{lemma}

\begin{proof}
    Add $\nabla f(y)$ to~\eqref{eq:bpg-optimality}.
\end{proof}

\subsection{Backtracking}
\label{app:common-backtracking}

\begin{proposition}[Uniform stepsize threshold]
    \label{prop:local-rs-backtracking}
    The upper-bound test holds at every admissible base point whenever $\lambda\leq1/L$.
    Thus geometric reductions by $\gamma_->1$ terminate when the base point is fixed.
    They also permit only finitely many stepsize reductions along any search in which the base point changes but the stepsize is never increased.
\end{proposition}

\begin{proof}
    For $y\in T_\lambda(x)$, relative smoothness gives
    \begin{equation*}
        D_f(y,x)
        \leq L D_\psi(y,x)
        \leq \lambda^{-1}D_\psi(y,x)
    \end{equation*}
    for $\lambda\leq1/L$.
    The threshold is independent of $x$.
    Geometric reduction reaches it in finitely many steps, and no subsequent reduction can then be required.
\end{proof}

\begin{proposition}[Accepted stepsizes and backtracking count]
    \label{prop:accepted-lambda-lower}
    In GA-BPGc and GA-BPGsc, initialize the first search at $\lambda_0$ and each subsequent search at $\gamma_+\lambda_{k-1}$.
    In GA-BPGnc, initialize at $\lambda_0$ for $k=1$ and at
    \begin{equation*}
        \widetilde\lambda_k
        =\min\{\gamma_+\lambda_{k-1},\lambda_{\max}\}
    \end{equation*}
    for $k\geq2$, with $\gamma_+\geq1$ and $\lambda_{\max}\geq\lambda_0$.
    Then every accepted stepsize is at least $\underline\lambda=\min\{\lambda_0,1/(\gamma_-L)\}$.
    For GA-BPGnc, if $m_k$ counts failed upper-bound tests,
    \begin{equation}
        \sum_{k=1}^N m_k
        \leq
        \frac{
            (N-1)\log\gamma_+
            +\log(\lambda_0/\underline\lambda)
        }{\log\gamma_-}.
        \label{eq:total-backtracking}
    \end{equation}
\end{proposition}

\begin{proof}
    An initial trial stepsize is no smaller than the preceding accepted value.
    Whenever a reduction occurs, its preceding value must exceed $1/L$; hence the reduced value exceeds $1/(\gamma_-L)$.
    Induction proves the lower bound, including when the base point in GA-BPGc and GA-BPGsc is recomputed after a reduction.

    In GA-BPGnc, $\lambda_1=\lambda_0\gamma_-^{-m_1}$ and $\lambda_k\leq\gamma_+\lambda_{k-1}\gamma_-^{-m_k}$ for $k\geq2$.
    Multiplication and the lower bound on $\lambda_N$ yield~\eqref{eq:total-backtracking}.
\end{proof}

The count in~\eqref{eq:total-backtracking} covers failed upper-bound tests.
In GA-BPGc and GA-BPGsc, trials rejected by a $\kappa$-dependent test can require additional BPG solves; a bound on their number must be established separately.

\subsection{Coefficient and mirror identities}
\label{app:common-identities}

The following results are used for GA-BPGc and GA-BPGsc, with $\phi=\psi$ in the latter.

\begin{lemma}[Coefficients and coupling]
    \label{lem:coefficients}
    For $\omega_-\geq0$ and $\lambda,\kappa>0$, the convex coefficient equation has the unique positive solution
    \begin{equation*}
        \alpha=
        \frac{\lambda+\sqrt{\lambda^2+2\kappa\lambda \omega_-}}{\kappa}.
    \end{equation*}
    For $\theta_-=1+\mu \omega_-$ and $\mu>0$, the strong-regime equation has a positive solution if and only if $d\coloneq\kappa-2\mu\lambda>0$.
    That solution is unique and equals
    \begin{equation*}
        \alpha=
        \frac{
            \lambda(\theta_-+\mu \omega_-)
            +\sqrt{
                \lambda^2(\theta_-+\mu \omega_-)^2+2\lambda d\omega_-\theta_-
            }
        }{d}.
    \end{equation*}
    For fixed $\lambda,\omega_-,\theta_-$, both roots tend to zero as $\kappa\to\infty$.
    If $\omega=\omega_-+\alpha$ and $\omega x=\omega_-y_-+\alpha z_-$, then
    \begin{equation*}
        \omega_-(x-y_-)+\alpha(x-u)=\alpha(z_--u).
    \end{equation*}
\end{lemma}

\begin{proof}
    The respective quadratic equations are
    \begin{equation*}
        \kappa \alpha^2-2\lambda \alpha-2\lambda \omega_-=0
    \end{equation*}
    and
    \begin{equation*}
        d \alpha^2-2\lambda(\theta_-+\mu \omega_-)\alpha-2\lambda \omega_-\theta_-=0.
    \end{equation*}
    The quadratic formula proves the expressions and uniqueness of the positive roots, including $\omega_-=0$.
    When $d\leq0$, the second polynomial is strictly negative for every $\alpha>0$.
    The formulas give the limits as $\kappa\to\infty$.
    The coupling identity follows by expanding its left side.
\end{proof}

\begin{lemma}[Mirror identity]
    \label{lem:mirror-identity}
    Suppose $z_-,z\in\interior\dom \phi$ and $\nabla\phi(z)=\nabla\phi(z_-)-\alpha g$.
    For $u\in\dom \phi$,
    \begin{equation*}
        \alpha\langle g,z_--u\rangle
        +D_\phi(u,z)-D_\phi(u,z_-)
        =
        D_\phi(z_-,z).
    \end{equation*}
\end{lemma}

\begin{proof}
    Expanding the divergences gives
    \begin{equation*}
        D_\phi(u,z)-D_\phi(u,z_-)
        =
        \alpha\langle g,u-z\rangle-D_\phi(z,z_-).
    \end{equation*}
    Also,
    \begin{equation*}
        \alpha\langle g,z_--z\rangle
        =
        D_\phi(z_-,z)+D_\phi(z,z_-).
    \end{equation*}
    Add the identities.
\end{proof}

\begin{lemma}[Weighted mirror identity]
    \label{lem:weighted-mirror-identity}
    Suppose $\theta=\theta_-+\mu \alpha$ and $z_-,x,z\in\Omega$ satisfy
    \begin{equation*}
        \theta\nabla\psi(z)
        =
        \theta_-\nabla\psi(z_-)+\mu \alpha\nabla\psi(x)-\alpha g.
    \end{equation*}
    For $u\in\dom \psi$,
    \begin{align}
        \theta D_\psi(u,z)-\theta_-D_\psi(u,z_-)
        &=
        \alpha\langle g,u-z\rangle+\mu \alpha D_\psi(u,x)
        \notag\\
        &\quad
        -\mu \alpha D_\psi(z,x)-\theta_-D_\psi(z,z_-),
        \label{eq:weighted-mirror-identity}
        \\
        \alpha\langle g,z_--z\rangle
        -\mu \alpha D_\psi(z,x)-\theta_-D_\psi(z,z_-)
        &=
        \theta D_\psi(z_-,z)-\mu \alpha D_\psi(z_-,x).
        \label{eq:weighted-mirror-remainder}
    \end{align}
\end{lemma}

\begin{proof}
    Expand $\theta D_\psi(u,z)-\theta_-D_\psi(u,z_-)$.
    The coefficient of $\psi(u)$ is $\mu \alpha$, and the coefficient of $u$ is $\alpha g-\mu \alpha\nabla\psi(x)$.
    These agree with the first two terms on the right of \eqref{eq:weighted-mirror-identity}; expanding the remaining constant terms gives the other two terms.
    Set $u=z_-$ in that identity to obtain \eqref{eq:weighted-mirror-remainder}.
\end{proof}

\begin{lemma}[Small mirror displacement]
    \label{lem:small-mirror}
    Let $\phi$ be Legendre, $z_-\in\interior\dom\phi$, and $p=\nabla\phi(z_-)$.
    If $p_\alpha=p+\alpha v_\alpha$, where $\alpha\downarrow0$ and $\{v_\alpha\}$ is bounded, then eventually $p_\alpha\in\interior\dom\phi^*$, $z_\alpha=\nabla\phi^*(p_\alpha)\to z_-$, and $D_\phi(z_-,z_\alpha)=o(\alpha)$.
\end{lemma}

\begin{proof}
    The conjugate domain is open around $p$, and $\nabla\phi^*$ is continuous there.
    Legendre duality gives $D_\phi(z_-,z_\alpha)=D_{\phi^*}(p_\alpha,p)$.
    Differentiability of $\phi^*$ at $p$ implies $D_{\phi^*}(p_\alpha,p)=o(\|p_\alpha-p\|)=o(\alpha)$; the case $p_\alpha=p$ is immediate.
\end{proof}

\section{Convergence analysis for convex objectives}
\label{app:convex}

This section provides the convergence analysis of GA-BPGc.

\subsection{One-step estimate}
\label{app:convex-one-step-proof}

For this and the remaining subsections, we consider Algorithm~\ref{alg:convex-full} under Assumptions~\ref{ass:standing}, \ref{assumption:subproblem}, and~\ref{assumption:convex}, with
\begin{equation*}
    y_0=z_0\in
    \Omega\cap\dom\rho\cap\interior\dom\phi,
    \qquad \omega_0=0.
\end{equation*}
No comparison inequality between $D_\psi$ and $D_\phi$ is imposed.
Comparison points belong to $\mathcal U_\phi=\cl\Omega\cap\dom\Phi\cap\dom\phi$.

\begin{proof}[Proof of Lemma~\ref{lem:convex-one-step}]
    Write $\omega_-=\omega_{k-1}$, $y_-=y_{k-1}$, and $z_-=z_{k-1}$.
    Apply Lemma~\ref{lem:bpg-comparison} at the same trial $(x,\lambda)$ to $y_-$ and $u$:
    \begin{align*}
        \Phi(\widehat y)-\Phi(y_-)
        &\leq
        \langle g_\lambda(x),x-y_-\rangle-P_\lambda(x)-D_f(y_-,x),
        \\
        \Phi(\widehat y)-\Phi(u)
        &\leq
        \langle g_\lambda(x),x-u\rangle-P_\lambda(x)-D_f(u,x).
    \end{align*}
    Multiply these inequalities by $\omega_-$ and $\alpha$, respectively.
    Since $\omega=\omega_-+\alpha$, the coupling identity gives
    \begin{align}
        &\omega(\Phi(\widehat y)-\Phi(u))
        -\omega_-(\Phi(y_-)-\Phi(u))
        \notag\\
        &\qquad\leq
        \alpha\langle g_\lambda(x),z_--u\rangle-\omega P_\lambda(x)
        -\omega_-D_f(y_-,x)-\alpha D_f(u,x).
        \label{eq:convex-objective-step}
    \end{align}
    The mirror equation and Lemma~\ref{lem:mirror-identity} imply
    \begin{equation*}
        \alpha\langle g_\lambda(x),z_--u\rangle
        +D_\phi(u,\widehat z)-D_\phi(u,z_-)
        =
        D_\phi(z_-,\widehat z).
    \end{equation*}
    Adding $D_\phi(u,\widehat z)-D_\phi(u,z_-)$ to both sides of \eqref{eq:convex-objective-step} and using the mirror identity above gives the corresponding trial estimate.
    For an accepted trial, $(\omega,\alpha,x,\widehat y,\widehat z,\lambda) =(\omega_k,\alpha_k,x_k,y_k,z_k,\lambda_k)$, so the left-hand side becomes $\mathcal E_k^{\mathrm C}(u)-\mathcal E_{k-1}^{\mathrm C}(u)$.
    This proves~\eqref{eq:convex-one-step-main}.

    Convexity makes $D_f(u,x)\geq0$, and the acceptance test gives
    \begin{equation*}
        D_\phi(z_-,\widehat z)-\omega P_\lambda(x)-\omega_-D_f(y_-,x)\leq0.
    \end{equation*}
    Hence every accepted iteration decreases $\mathcal E_k^{\mathrm C}(u)$.
\end{proof}

\subsection{Telescoping and weight growth}
\label{app:convex-posterior-proof}
\label{app:convex-analysis}

\begin{proof}[Proof of Lemma~\ref{lem:convex-weight-growth}]
    For every accepted iteration,
    \begin{equation*}
        \sqrt{\omega_k}-\sqrt{\omega_{k-1}}
        =
        \frac{\alpha_k}{\sqrt{\omega_k}+\sqrt{\omega_{k-1}}}
        \geq
        \frac{\alpha_k}{2\sqrt{\omega_k}}
        =
        \sqrt{\frac{\lambda_k}{2\kappa_k}}.
    \end{equation*}
    Sum from $k=1$ to $N$, use $\omega_0=0$, and square the resulting nonnegative inequality.
\end{proof}

\begin{proof}[Proof of Theorem~\ref{thm:convex-posterior}]
    Telescoping gives $\mathcal E_N^{\mathrm C}(u) \leq\mathcal E_0^{\mathrm C}(u)=D_\phi(u,z_0)$.
    Equivalently,
    \begin{equation}
        \omega_N(\Phi(y_N)-\Phi(u))+D_\phi(u,z_N)
        \leq D_\phi(u,z_0).
        \label{eq:convex-telescope}
    \end{equation}
    Since $\omega_N>0$, dropping $D_\phi(u,z_N)\geq0$ and using Lemma~\ref{lem:convex-weight-growth} proves~\eqref{eq:convex-posterior}.
    If $u=x^\star$ is a minimizer with finite $\phi(x^\star)$ and $\lambda_k/\kappa_k\geq c>0$, its denominator is at least $cN^2$.
    This proves the stated conditional accelerated rate.
\end{proof}

The bound is meaningful for any finite comparator, not just a minimizer.
If every minimizer lies outside $\dom \phi$, substituting a minimizer gives no finite constant.
An approximation argument would then require separate control of $D_\phi(u,z_0)$ along the chosen approximations.
Likewise, the bound alone does not establish that $\sum_{k=1}^N\sqrt{\lambda_k/\kappa_k}$ diverges as $N\to\infty$.

\subsection{Finite termination and the stopping convention}
\label{app:convex-finite}

\begin{proposition}[Finite inner search at a nonstationary iterate]
    \label{prop:convex-finite}
    Fix an outer iterate $(\omega_-,y_-,z_-)$ of GA-BPGc with $0\notin\nabla f(y_-)+\partial\rho(y_-)$.
    If $\omega_-=0$, take $y_-=z_-$ as in the initialization.
    Then the inner search accepts a trial or returns a stationary point after finitely many trials.
\end{proposition}

\begin{proof}
    Proposition~\ref{prop:local-rs-backtracking} permits only finitely many reductions of $\lambda$, even though coupling changes $x$.
    On the tail of a hypothetical infinite search, $\lambda=\bar\lambda>0$ is fixed and every rejection increases $\kappa$.
    Thus $\kappa\to\infty$ and $\alpha\to0$ by Lemma~\ref{lem:coefficients}.

    If $\omega_->0$, then $x\to y_-$.
    Let $\bar y\in T_{\bar\lambda}(y_-)$.
    Proposition~\ref{prop:bpg-continuity} implies $\widehat y\to\bar y$, and $g_\lambda(x)$ remains bounded.
    The upper-bound tests on this tail and the lower bound on $P_\lambda(x)$ give
    \begin{equation*}
        P_\lambda(x)\longrightarrow P_{\bar\lambda}(y_-)
        \geq
        \bar\lambda^{-1}
        D_\psi(y_-,\bar y)
        >0.
    \end{equation*}
    The strict inequality uses the nonstationarity of $y_-$.
    The mirror argument tends to $\nabla\phi(z_-)$, so its domain check passes eventually and $\widehat z\to z_-$.
    Since $z_-\in\Omega$ and $\Omega$ is open, feasibility also holds eventually.
    Then $D_\phi(z_-,\widehat z)\to0$ whereas $\omega P_\lambda(x)\to \omega_-P_{\bar\lambda}(y_-)>0$.
    Since $D_f(y_-,x)\geq0$, the acceptance inequality eventually holds, a contradiction.

    If $\omega_-=0$, then $x=y_-=z_-$ and $\omega=\alpha$.
    For the fixed $\bar\lambda$, $g_\lambda(x)$ and $P_\lambda(x)>0$ are fixed.
    Lemma~\ref{lem:small-mirror} gives $D_\phi(z_-,\widehat z)=o(\alpha)$, whereas $\omega P_\lambda(x)=\alpha P_\lambda(x)$.
    The domain and feasibility checks pass eventually, and so does the acceptance inequality, again a contradiction.
\end{proof}

\paragraph{Already stationary previous iterates.}
Proposition~\ref{prop:convex-finite} does not by itself cover a stationary previous iterate, because its strict lower bound on the limiting value of $P_\lambda(x)$ then vanishes.

The stopping convention in Section~\ref{sec:setting-methods} includes a previous-iterate check for GA-BPGc and GA-BPGsc.
With $\lambda$ initialized at the preceding accepted stepsize, they compute $v\in T_\lambda(y_-)$ and return $y_-$ if $v=y_-$.
By~\eqref{eq:bpg-fixed-point}, this detects stationarity of $y_-$.
If the check does not terminate, the previous iterate is nonstationary and the finite-search argument applies.

The check adds at most one BPG solve per outer iteration and leaves the trial sequence unchanged whenever the previous iterate is nonstationary.

\begin{corollary}[Finite termination with the stopping check]
    \label{cor:convex-finite-guard}
    With this check, each outer iteration of GA-BPGc accepts a trial or returns a global minimizer after finitely many BPG solves.
    Without the check, Proposition~\ref{prop:convex-finite} and all bounds for completed iterations remain valid with their stated qualifications.
\end{corollary}

\begin{proof}
    By~\eqref{eq:bpg-fixed-point}, the check detects stationarity of $y_-$.
    An interior stationary point minimizes $\Phi$ over $\cl\Omega$: add the convex supporting inequality for $f$ at that point to the supporting inequality for $\rho$ supplied by stationarity.
    If the check does not terminate, Proposition~\ref{prop:convex-finite} applies.
    The same argument shows that an in-trial stationary point is a global minimizer.
\end{proof}

\subsection{Sufficient conditions for convex acceleration}
\label{app:convex-sublevel-results}
\begin{lemma}[Existence of compact subsets of the objective sublevel]
    \label{lem:c5-compact-trials}
    Under the assumptions of Theorem~\ref{thm:convex-kernel-certificate}, all accepted points $y_k,z_k$, all trial points $x$, and all trial BPG points $\widehat y$ satisfying \eqref{eq:local-rs-test} belong to one fixed compact subset of $\Omega$.
\end{lemma}

\begin{proof}
    By condition~(iv), \citet[Theorems~3.7(vi) and~3.8(i)]{bauschke1997legendre}, and continuity, the sublevel sets of $D_\psi(u,\cdot)$ and $D_\phi(u,\cdot)$ are compact in $\Omega$.
    Applying \eqref{eq:convex-telescope} with comparison points $y_0,u$ and using $z_0=y_0$ gives
    \begin{align}
        &\omega_k\left(\Phi(y_k)-\Phi(y_0)\right)+D_\phi(y_0,z_k) \leq0,
        \label{eq:c5-initial-sublevel}\\
        &\omega_k\left(\Phi(y_k)-\Phi^\star\right)+D_\phi(u,z_k)
        \leq D_\phi(u,y_0).
        \label{eq:c5-energy-bound}
    \end{align}
    Thus $\Phi(y_k)\leq\Phi(y_0)$, and all $z_k$ lie in the compact set $\mathcal{K}\coloneq\{z\in\Omega:D_\phi(u,z)\leq D_\phi(u,y_0)\}$.

    Fix a closed ball of positive radius centered at $u$ and contained in $\Omega$.
    Convexity of $\psi$ gives $\langle\nabla\psi(y),v-u\rangle\leq D_\psi(u,y)+\psi(v)-\psi(u)$ for every $v$ in this ball.
    The maximum of the left-hand side over the ball equals its radius times $\|\nabla\psi(y)\|$, while $\psi(v)-\psi(u)$ is bounded above on the ball by continuity.
    Hence $\|\nabla\psi(y)\|\leq C(1+D_\psi(u,y))$ for some $C>0$ independent of $y\in\Omega$.
    Similarly, enlarging $C$ if necessary, we obtain
    \begin{equation*}
        D_\phi(z,y)\leq C\left(1+D_\phi(y_0,y)\right),
        \qquad D_\psi(y_0,y)\leq C\left(1+D_\psi(u,y)\right),
    \end{equation*}
    uniformly for $z\in \mathcal{K}$ and $y\in\Omega$.
    If $\beta>0$, these comparisons and condition~(iii) imply $D_\phi(z,y)/D_\psi(u,y)\to0$ uniformly in $z\in \mathcal{K}$ as $y$ leaves compact subsets of $\Omega$ within the initial objective sublevel.
    Indeed, every such sequence has a subsequence approaching the boundary or escaping to infinity, and $D_\psi(u,y)\to\infty$.
    Since $D_\psi(u,z)$ is uniformly bounded for $z\in \mathcal{K}$, the limit above implies $D_\psi(u,z)+\beta D_\phi(z,y)\le \tfrac12D_\psi(u,y)$ whenever $D_\psi(u,y)$ is sufficiently large.
    When $D_\psi(u,y)$ is below this fixed threshold, $y$ lies in a compact sublevel set of $D_\psi(u,\cdot)$, so continuity gives a uniform upper bound for the left-hand side over these $y$ and all $z\in \mathcal{K}$.
    Hence there is $C>0$ such that
    \begin{equation}
        D_\psi(u,z)+\beta D_\phi(z,y)\leq\tfrac12D_\psi(u,y)+C
        \qquad\left(z\in \mathcal{K},\ \Phi(y)\leq\Phi(y_0),\ y\in\Omega\right).
        \label{eq:c5-uniform-domination}
    \end{equation}
    For $\beta=0$, the same bound follows directly from compactness of $\mathcal{K}$.

    For any trial, condition (ii), convexity of $\phi$, and $x=(\omega_-y+\alpha z)/\omega$ give
    \begin{align}
        D_\psi(u,x)
        &\leq\frac{\omega_-}{\omega}D_\psi(u,y)
        +\frac{\alpha}{\omega}\left(D_\psi(u,z)+\beta D_\phi(z,y)\right)\notag\\
        &\leq\left(1-\frac{\alpha}{2\omega}\right)D_\psi(u,y)
        +\frac{\alpha C}{\omega}
        \leq\max\{D_\psi(u,y),2C\}.
        \label{eq:c5-coupling-bound}
    \end{align}
    Whenever \eqref{eq:local-rs-test} holds, Lemma~\ref{lem:bpg-comparison} and the three-point identity give
    \begin{equation}
        D_\psi(u,\widehat y)+\lambda\left(\Phi(\widehat y)-\Phi^\star\right)
        \leq D_\psi(u,x),
        \label{eq:burg-bpg-comparison}
    \end{equation}
    so $D_\psi(u,\widehat y)\leq D_\psi(u,x)$.
    Induction therefore bounds $D_\psi(u,y_k)$, $D_\psi(u,x)$, and $D_\psi(u,\widehat y)$ by $\max\{D_\psi(u,y_0),2C\}$, where the bound for $x$ does not require \eqref{eq:local-rs-test}.
    Compactness of the corresponding Bregman sublevel proves the claim.
\end{proof}

\begin{proof}[Proof of Theorem~\ref{thm:convex-kernel-certificate}]
    Choose a fixed compact convex $\mathcal{K}\subset\Omega$ whose interior contains the points from Lemma~\ref{lem:c5-compact-trials} and the compact set $\{w\in\Omega:D_\phi(u,w)\leq D_\phi(u,y_0)\}$.
    Such a set exists because $\Omega$ is open and convex.
    For a trial satisfying \eqref{eq:local-rs-test}, first define
    \begin{equation*}
        \widehat z=\argmin_{w\in \mathcal{K}}
        \left\{D_\phi(w,z)+\alpha\langle g_\lambda(x),w-z\rangle\right\}.
    \end{equation*}
    Compactness, continuity, and strict convexity ensure existence and uniqueness.
    Its variational inequality and the three-point identity yield, for every $v\in \mathcal{K}$,
    \begin{align}
        &\alpha\langle g_\lambda(x),z-v\rangle
        +D_\phi(v,\widehat z)-D_\phi(v,z)\notag\\
        &\quad\leq\alpha\langle g_\lambda(x),z-\widehat z\rangle
        -D_\phi(\widehat z,z)\notag\\
        &\quad\leq C_{\mathcal{K}}\sqrt{
            \frac{2\omega P_\lambda(x)D_\phi(\widehat z,z)}{\kappa}}
        -D_\phi(\widehat z,z)
        \leq\frac{C_{\mathcal{K}}^2}{2\kappa}\omega P_\lambda(x).
        \label{eq:c5-constrained-mirror}
    \end{align}
    The last two inequalities use condition~(v), $D_\psi(x,\widehat y)\leq\lambda P_\lambda(x)$ from Lemma~\ref{lem:bpg-comparison}, and $\kappa\alpha^2=2\lambda\omega$.

    Suppose $\kappa\geq C_{\mathcal{K}}^2/2$.
    Taking $v=u$ in \eqref{eq:c5-constrained-mirror} and combining it with \eqref{eq:convex-objective-step} gives
    \begin{align*}
        \omega\left(\Phi(\widehat y)-\Phi^\star\right)+D_\phi(u,\widehat z)
        &\leq\omega_-\left(\Phi(y)-\Phi^\star\right)+D_\phi(u,z)\\
        &\quad-\left(1-\frac{C_{\mathcal{K}}^2}{2\kappa}\right)\omega P_\lambda(x)
        \leq D_\phi(u,y_0),
    \end{align*}
    using convexity, $P_\lambda(x)\geq0$, and \eqref{eq:c5-energy-bound}.
    Since $\Phi(\widehat y)\geq\Phi^\star$, $D_\phi(u,\widehat z)\leq D_\phi(u,y_0)$, so $\widehat z\in\interior \mathcal{K}$.
    The constraint is therefore inactive, and optimality gives $\nabla\phi(\widehat z)=\nabla\phi(z)-\alpha g_\lambda(x)=p$.
    Legendre duality implies $p\in\interior\dom\phi^*$ and $\widehat z=\nabla\phi^*(p)\in\Omega$.
    Thus the constructed point is the algorithm's mirror candidate, and both domain checks pass.
    Taking $v=z$ in \eqref{eq:c5-constrained-mirror} also gives
    \begin{equation}
        D_\phi(z,\widehat z)\leq\frac{C_{\mathcal{K}}^2}{2\kappa}\omega P_\lambda(x)
        \leq\omega P_\lambda(x).
        \label{eq:c5-uniform-certificate}
    \end{equation}
    Since $D_f(y,x)\geq0$, the acceptance test \eqref{eq:convex-certificate} passes.
    The fixed set $\mathcal{K}$ makes $C_{\mathcal{K}}^2/2$ a uniform acceptance threshold.

    Together with the uniform upper-bound threshold $\lambda\leq1/L$, this ensures finite inner searches.
    Each increase of $\kappa$ starts below the acceptance threshold and multiplies it by $\gamma_\kappa$, while warm starts never increase $\kappa$.
    Hence $\kappa_k\leq\max\{\kappa_0,\gamma_\kappa C_{\mathcal{K}}^2/2\}$ for every $k\geq1$.
    Combining this with Proposition~\ref{prop:accepted-lambda-lower} gives the explicit ratio bound
    \begin{equation}
        \frac{\lambda_k}{\kappa_k}
        \geq
        \frac{\min\{\lambda_0,1/(\gamma_-L)\}}
        {\max\{\kappa_0,\gamma_\kappa C_{\mathcal{K}}^2/2\}}
        >0
        \qquad(k\geq1).
        \label{eq:c5-ratio-lower-bound}
    \end{equation}
    Therefore Theorem~\ref{thm:convex-posterior} yields $\Phi(y_N)-\Phi^\star=\O(N^{-2})$.
\end{proof}

Table~\ref{tab:convex-kernel-conditions} summarizes the applicability of Theorem~\ref{thm:convex-kernel-certificate}.
The following corollary illustrates how to verify the assumptions of Theorem~\ref{thm:convex-kernel-certificate} for the Burg--KL geometries.

\begin{table}[t]
    \centering
    \footnotesize
    \setlength{\tabcolsep}{3pt}
    \renewcommand{\arraystretch}{1.15}
    \caption{
        Kernel pairs for Theorem~\ref{thm:convex-kernel-certificate}.
        \cmark: conditions (i)--(v) hold without additional assumptions; \tmark: they hold under the listed additional assumptions on the initial objective sublevel set; \xmark: a theorem assumption fails, as indicated in the last column.
        Sublevel assumptions refer to the initial objective sublevel and are used only for condition~(iii).
        $\dagger$: condition~(iii) can also fail when this sublevel set is unbounded; $\ddagger$: condition~(ii) also fails when $u\neq0$.
    }
    \label{tab:convex-kernel-conditions}

    \renewcommand{\tabularxcolumn}[1]{m{#1}}
    \begin{tabularx}{\textwidth}{@{}llcXc@{}}
        \toprule
        $\psi$ & $\phi$ & Assumptions & Additional assumption & Failure \\
        \midrule

        $\frac12\|x\|^2$
        & $\frac12\|x\|^2$
        & \cmark & -- & -- \\
        \addlinespace

        $\sum_i(x_i\log x_i-x_i)$
        & $\sum_i(x_i\log x_i-x_i)$
        & \cmark & -- & -- \\
        \addlinespace

        $\sum_i(x_i\log x_i-x_i)$
        & $\frac12\|x\|^2$
        & \xmark & -- & (i) \\
        \addlinespace

        $-\sum_i\log x_i$
        & $\sum_i(x_i\log x_i-x_i)$
        & \tmark & Bounded sublevel & -- \\
        \addlinespace

        $-\sum_i\log x_i$
        & $\frac12\|x\|^2$
        & \xmark & -- & (i), (iii)$^\dagger$ \\
        \addlinespace

        $-\sum_i\log x_i$
        & $-\sum_i\log x_i$
        & \tmark & Compact sublevel & -- \\
        \addlinespace

        $-\sum_i\log(x_i(1-x_i))$
        & $\begin{aligned}
            &\sum_i[x_i\log x_i\\[-2pt]
            &\quad +(1-x_i)\log(1-x_i)]
        \end{aligned}$
        & \cmark & -- & -- \\
        \addlinespace

        $-\log\det X$
        & $\tr(X\log X-X)$
        & \tmark & Bounded sublevel & -- \\
        \addlinespace

        $-\log\det X$
        & $\frac12\|X\|_{\F}^2$
        & \xmark & -- & (i), (iii)$^\dagger$ \\
        \midrule

        $\frac1q\sum_i|x_i|^q\ (q\ne2)$
        & $\frac1q\sum_i|x_i|^q$
        & \xmark & -- & (v), (ii)$^\ddagger$ \\
        \addlinespace

        $\frac1q\sum_i|x_i|^q\ (q\geq3)$
        & $\frac12\|x\|^2$
        & \cmark & -- & -- \\
        \addlinespace

        $\frac14\|x\|^4+\frac12\|x\|^2$
        & $\frac14\|x\|^4+\frac12\|x\|^2$
        & \tmark & Bounded sublevel & -- \\
        \addlinespace

        $\frac14\|x\|^4+\frac12\|x\|^2$
        & $\frac12\|x\|^2$
        & \cmark & -- & -- \\
        \bottomrule
    \end{tabularx}
\end{table}

\begin{corollary}[Burg--KL acceleration]
    \label{cor:c5-burg-kl}
    Suppose that \cref{ass:standing,assumption:subproblem,assumption:convex} hold, $u\in\Omega$ minimizes $\Phi$ over $\cl\Omega$, and $y_0\in\Omega\cap\dom\Phi$.
    Let
    \begin{equation*}
        \Omega=\R_{++}^n,\qquad
        \psi(x)=-\sum_i\log x_i,\qquad
        \phi(x)=\sum_i(x_i\log x_i-x_i),
    \end{equation*}
    with their lower semicontinuous convex extensions.
    If $\{x\in\Omega:\Phi(x)\leq\Phi(y_0)\}$ is bounded, then GA-BPGc satisfies $\Phi(y_N)-\Phi^\star=\O(N^{-2})$.
\end{corollary}

\begin{proof}
    Both generators are Legendre with interior domain $\Omega$, verifying condition~(i).
    We check the conditions of Theorem~\ref{thm:convex-kernel-certificate}.

    For (ii), both Hessians are diagonal, and
    \begin{equation*}
        \frac{\partial^2}{\partial x_i^2}D_\psi(u,x)
        +\frac{1}{8u_i x_i}
        =\frac{(x_i-4u_i)^2}{8u_i x_i^3}\geq0.
    \end{equation*}
    Since $\frac{\partial^2\phi(x)}{\partial x_i^2}=1/x_i$, (ii) holds with $\beta=\max_i(8u_i)^{-1}$.

    For (iii), use
    \begin{align*}
        D_\psi(y_0,x)
        &=\sum_i\left(\frac{y_{0,i}}{x_i}-1+\log\frac{x_i}{y_{0,i}}\right),\\
        D_\phi(y_0,x)
        &=\sum_i\left(y_{0,i}\log\frac{y_{0,i}}{x_i}-y_{0,i}+x_i\right).
    \end{align*}
    Along $x_j\to v\in\partial\Omega$, coordinates with $v_i>0$ contribute bounded terms.
    For $v_i=0$, the ratio of the entropy term to the barrier term tends to zero.
    At least one barrier term diverges, and all are nonnegative, so $D_\phi(y_0,x_j)/D_\psi(y_0,x_j)\longrightarrow0$.
    Boundedness of the initial objective sublevel excludes escape to infinity, proving (iii).

    For (iv), $\dom\psi^*=(-\infty,0)^n$ and $\dom\phi^*=\R^n$ are open.
    Finally, $\nabla^2\psi(x)=\Diag(x_i^{-2})$ and $\nabla^2\phi(x)=\Diag(x_i^{-1})$ are continuous and positive definite on $\Omega$, proving (v) by the Hessian criterion.
    The theorem gives the stated rate.
\end{proof}

\section{Convergence analysis for relatively strongly convex objectives}
\label{app:strong}

We impose the relatively strongly convex condition of Assumption~\ref{ass:strong}, with known $\mu>0$.
Comparison points belong to $\mathcal U_\psi=\cl\Omega\cap\dom \Phi\cap\dom \psi$.
The BPG and mirror updates both use $\psi$.
The stopping convention in Section~\ref{sec:setting-methods} applies; its tests are omitted from the pseudocode.

\subsection{Algorithm}
\label{app:strong-algorithm}

\begin{algorithm}[t]
    \caption{GA-BPGsc}
    \label{alg:strong-full}
    \SetKwFunction{BT}{Backtracking}

    \KwIn{
        Known $\mu>0$;
        $y_0\in\Omega\cap\dom \rho$;
        $\lambda_0,\kappa_0>0$;
        $\gamma_+\geq1$; $\gamma_-,\gamma_\kappa>1$.
    }

    $z_0\gets y_0$, $\omega_0\gets0$, $\theta_0\gets1$\;

    \For{$k=1,2,\ldots$}{
    $(\widehat\lambda,\widehat\kappa)\gets(\gamma_+\lambda_{k-1},\kappa_{k-1}/\gamma_+)$\;
    $(\alpha_k,\omega_k,\theta_k,x_k,y_k,z_k,g_k,\lambda_k,\kappa_k)\gets$
    \BT{$\omega_{k-1},\theta_{k-1},y_{k-1},z_{k-1},\widehat\lambda,\widehat\kappa$}\;
    }

    \Procedure{\BT{$\omega_-,\theta_-,y,z,\lambda,\kappa$}}{
        \If(\Comment*[f]{Criterion}){$\kappa\leq2\mu\lambda$}{
            \Return{\BT{$\omega_-,\theta_-,y,z,\lambda,\gamma_\kappa\kappa$}}\;
        }

        $d\gets\kappa-2\mu\lambda$,
        $\alpha\gets
            [\lambda(\theta_-+\mu \omega_-)+\sqrt{\lambda^2(\theta_-+\mu \omega_-)^2+2\lambda d\omega_-\theta_-}]/d$
        \Comment*{Coupling}

        $\omega\gets \omega_-+\alpha$,
        $\theta\gets \theta_-+\mu \alpha$,
        $x\gets(\omega_-y+\alpha z)/\omega$\;

        $\widehat y\in T_\lambda(x)$
        \Comment*{BPG update}

        \If{$D_f(\widehat y,x)>D_\psi(\widehat y,x)/\lambda$}{
            \Return{\BT{$\omega_-,\theta_-,y,z,\lambda/\gamma_-,\kappa$}}\;
        }

        $g_\lambda(x)\gets
            [\nabla\psi(x)-\nabla\psi(\widehat y)]/\lambda$\;
        $P_\lambda(x)\gets
            [D_\psi(x,\widehat y)+D_\psi(\widehat y,x)]/\lambda-D_f(\widehat y,x)$\;
        $p\gets
            [\theta_-\nabla\psi(z)+\mu \alpha\nabla\psi(x)-\alpha g_\lambda(x)]/\theta$\;

        \If{$p\notin\interior\dom \psi^*$}{
            \Return{\BT{$\omega_-,\theta_-,y,z,\lambda,\gamma_\kappa\kappa$}}\;
        }

        $\widehat z\gets\nabla\psi^*(p)$
        \Comment*{Mirror update}

        \If(\Comment*[f]{Criterion}){$\theta D_\psi(z,\widehat z)-\mu(\omega_-D_\psi(y,x)+\alpha D_\psi(z,x))>\omega P_\lambda(x)$}{
            \Return{\BT{$\omega_-,\theta_-,y,z,\lambda,\gamma_\kappa\kappa$}}\;
        }

        \Return{$(\alpha,\omega,\theta,x,\widehat y,\widehat z,g_\lambda(x),\lambda,\kappa)$}\;
    }
\end{algorithm}

The acceptance test is given by~\eqref{eq:strong-certificate}.
The conjugate-domain check is performed before the inverse mirror map; that map returns a point in $\Omega$.
A rejected trial retains the previous accepted points and weights, as well as any changes to the trial $\lambda$ and $\kappa$.
Each recursive call retries with these updated trial parameters.
Induction gives $\theta_k=1+\mu \omega_k$.

\subsection{Lyapunov decrease and the product bound}
\label{app:strong-analysis}

We use the Lyapunov function $\mathcal E_k^{\mathrm S}(u)$ defined in~\eqref{eq:strong-energy}.

\begin{proof}[Proof of Lemma~\ref{lem:strong-one-step}]
    Consider a trial satisfying the coupling and mirror-update equations of GA-BPGsc, with $\widehat y\in T_\lambda(x)$ and a valid mirror candidate $\widehat z\in\Omega$, and set $\omega=\omega_{k-1}+\alpha$ and $\theta=\theta_{k-1}+\mu \alpha$.
    Apply~\eqref{eq:bpg-relative-lower} with comparison points $y_{k-1}$ and $u$, multiply by $\omega_{k-1}$ and $\alpha$, respectively, and add.
    The coupling identity gives
    \begin{align*}
        &\omega(\Phi(\widehat y)-\Phi(u))
        -\omega_{k-1}(\Phi(y_{k-1})-\Phi(u))
        \\
        &\quad\leq
        \alpha\langle g_\lambda(x),z_{k-1}-u\rangle-\omega P_\lambda(x)
        -\mu \omega_{k-1}D_\psi(y_{k-1},x)
        -\mu \alpha D_\psi(u,x).
    \end{align*}
    The weighted mirror identities \eqref{eq:weighted-mirror-identity} and~\eqref{eq:weighted-mirror-remainder} imply
    \begin{align*}
        &\alpha\langle g_\lambda(x),z_{k-1}-u\rangle
        +\theta D_\psi(u,\widehat z)
        -\theta_{k-1}D_\psi(u,z_{k-1})
        \\
        &\quad=
        \mu \alpha D_\psi(u,x)
        +\theta D_\psi(z_{k-1},\widehat z)
        -\mu \alpha D_\psi(z_{k-1},x).
    \end{align*}
    Adding the change in weighted Bregman divergence to the objective estimate cancels the terms containing $D_\psi(u,x)$.
    For an accepted trial, $(\omega,\theta,\alpha,x,\widehat y,\widehat z,\lambda) =(\omega_k,\theta_k,\alpha_k,x_k,y_k,z_k,\lambda_k)$, so the left-hand side becomes $\mathcal E_k^{\mathrm S}(u)-\mathcal E_{k-1}^{\mathrm S}(u)$, which yields~\eqref{eq:strong-one-step}.
\end{proof}

\begin{proof}[Proof of~\eqref{eq:strong-A-growth}]
    The coefficient condition gives $\kappa_k>2\mu\lambda_k$, hence $\sqrt{2\mu\lambda_k/\kappa_k}\in(0,1)$.
    Since $\theta_k=1+\mu \omega_k\geq\mu \omega_k$, the coefficient equation implies
    \begin{equation*}
        \alpha_k
        =
        \sqrt{\frac{2\lambda_k\omega_k\theta_k}{\kappa_k}}
        \geq
        \sqrt{\frac{2\mu\lambda_k}{\kappa_k}}\,\omega_k.
    \end{equation*}
    Therefore
    \begin{equation*}
        \omega_{k-1}=\omega_k-\alpha_k
        \leq\left(1-\sqrt{\frac{2\mu\lambda_k}{\kappa_k}}\right)\omega_k.
    \end{equation*}
    Iterating this inequality from $k=2$ to $N$ proves~\eqref{eq:strong-A-growth}.
    For $N=1$, the assertion holds with an empty product.
\end{proof}

\begin{proof}[Proof of Theorem~\ref{thm:strong-main}]
    Fix $u\in\mathcal U_\psi$.
    Lemma~\ref{lem:strong-one-step} and the acceptance test give
    \begin{equation*}
        \mathcal E_k^{\mathrm S}(u)
        \leq\mathcal E_{k-1}^{\mathrm S}(u)
    \end{equation*}
    at every completed iteration.
    Since $\omega_0=0$ and $\theta_0=1$, telescoping yields
    \begin{equation*}
        \mathcal E_N^{\mathrm S}(u)
        \leq\mathcal E_0^{\mathrm S}(u)
        =D_\psi(u,z_0).
    \end{equation*}
    Equivalently,
    \begin{equation}
        \omega_N(\Phi(y_N)-\Phi(u))
        +\theta_ND_\psi(u,z_N)
        \leq D_\psi(u,z_0).
        \label{eq:strong-telescope}
    \end{equation}
    Dropping $\theta_ND_\psi(u,z_N)\geq0$ and using \eqref{eq:strong-A-growth}, we obtain
    \begin{equation*}
        \Phi(y_N)-\Phi(u)
        \leq
        \frac{D_\psi(u,z_0)}{\omega_N}
        \leq
        \frac{D_\psi(u,z_0)}{\omega_1}
        \prod_{k=2}^N\left(1-\sqrt{\frac{2\mu\lambda_k}{\kappa_k}}\right),
    \end{equation*}
    which proves~\eqref{eq:strong-product}.
    Taking $u=x^\star$ gives the optimality-gap bound.
    If $\sqrt{2\mu\lambda_k/\kappa_k}\geq q_{\min}>0$, then
    \begin{equation*}
        \prod_{k=2}^N\left(1-\sqrt{\frac{2\mu\lambda_k}{\kappa_k}}\right)
        \leq (1-q_{\min})^{N-1},
    \end{equation*}
    so the objective gap converges linearly.
\end{proof}

\begin{corollary}[Distance statements and boundary points]
    \label{cor:strong-distance}
    For a minimizer $x^\star\in\mathcal U_\psi$,
    \begin{equation}
        D_\psi(x^\star,z_N)
        \leq
        \frac{D_\psi(x^\star,y_0)}{\theta_N}.
        \label{eq:strong-aux-distance}
    \end{equation}
    If, in addition, $x^\star\in\Omega$, then
    \begin{equation}
        D_\psi(y_N,x^\star)
        \leq
        \frac{D_\psi(x^\star,y_0)}{\mu \omega_1}
        \prod_{k=2}^N\left(1-\sqrt{\frac{2\mu\lambda_k}{\kappa_k}}\right).
        \label{eq:strong-primal-distance}
    \end{equation}
\end{corollary}

\begin{proof}
    In~\eqref{eq:strong-telescope}, the objective gap is nonnegative for a minimizer.
    Dropping it proves~\eqref{eq:strong-aux-distance}.

    For an interior minimizer, first-order optimality gives $-\nabla f(x^\star)\in\partial\rho(x^\star)$.
    Add the supporting inequality for $\rho$ to relative strong convexity at $x^\star$ to obtain
    \begin{equation*}
        \Phi(y)-\Phi^\star\geq\mu D_\psi(y,x^\star).
    \end{equation*}
    The product bound gives~\eqref{eq:strong-primal-distance}.
\end{proof}

\subsection{Finite termination}
\label{app:strong-finite}

\begin{proposition}[Finite strong-regime search]
    \label{prop:strong-finite}
    At a nonstationary previous iterate, the inner search of Algorithm~\ref{alg:strong-full} terminates finitely.
    With its previous-iterate stopping check, every outer iteration therefore accepts a trial or returns a global minimizer after finitely many BPG solves.
\end{proposition}

\begin{proof}
    The stopping check detects a stationary $y_-=y_{k-1}$ by~\eqref{eq:bpg-fixed-point}.
    Such a point is a global minimizer by the relative lower bound and convexity of $\rho$.
    Otherwise $y_-$ is nonstationary.

    There can be only finitely many reductions of $\lambda$.
    On an infinite remaining tail, $\lambda=\bar\lambda$ is fixed and $\kappa\to\infty$.
    The coefficient condition eventually holds, and $\alpha\to0$, $\omega\to \omega_-$, $\theta\to \theta_-$.

    For $\omega_->0$, let $\bar y\in T_{\bar\lambda}(y_-)$.
    Then $x\to y_-$, and the continuity and positivity arguments in Appendix~\ref{app:convex-finite} give $P_\lambda(x)\to P_{\bar\lambda}(y_-)>0$.
    The mirror argument converges to $\nabla\psi(z_-)$, so it is eventually in the conjugate domain and $\widehat z\to z_-$.
    Since the subtracted divergences in \eqref{eq:strong-certificate} are nonnegative,
    \begin{align*}
        &\theta D_\psi(z_-,\widehat z)
        -\mu\left(\omega_-D_\psi(y_-,x)+\alpha D_\psi(z_-,x)\right)
        \\
        &\quad\leq \theta D_\psi(z_-,\widehat z)\to0,
    \end{align*}
    whereas $\omega P_\lambda(x)\to \omega_-P_{\bar\lambda}(y_-)>0$.
    Acceptance follows eventually.

    For $\omega_-=0$, initialization gives $x=y_-=z_-$ and $\theta_-=1$.
    The mirror argument is $\nabla\psi(z_-)-(\alpha/\theta)g_\lambda(x)$.
    Lemma~\ref{lem:small-mirror} with $\phi=\psi$ yields
    \begin{align*}
        &\theta D_\psi(z_-,\widehat z)
        -\mu\left(\omega_-D_\psi(y_-,x)+\alpha D_\psi(z_-,x)\right)
        \\
        &\quad=\theta D_\psi(z_-,\widehat z)=o(\alpha),
    \end{align*}
    while $\omega P_\lambda(x)=\alpha P_\lambda(x)>0$.
    The domain and acceptance tests again hold eventually.
    These contradictions prove finite termination.
\end{proof}

\subsection{Sufficient conditions for strong acceleration}
\label{app:sc-acceleration-conditions}

\begin{proof}[Proof of Theorem~\ref{thm:sc-acceleration-conditions}]
    Applying \eqref{eq:strong-telescope} with comparison points $y_0,u$, and using $z_0=y_0$ and $\theta_k\geq1$, gives
    \begin{equation*}
        \Phi(y_k)\leq\Phi(y_0),
        \qquad D_\psi(u,z_k)\leq D_\psi(u,y_0).
    \end{equation*}
    These are precisely the starting bounds in the proof of Lemma~\ref{lem:c5-compact-trials} with $\phi=\psi$.
    Although that lemma concerns GA-BPGc, the compactness argument that follows uses only these bounds, relative weak convexity, condition~(iii), coupling, and \eqref{eq:burg-bpg-comparison}; all remain valid here.
    In particular, the relative weak-convexity argument leading to \eqref{eq:c5-uniform-domination} and \eqref{eq:c5-coupling-bound} applies also when $\beta>0$.
    Thus all accepted $y_k,z_k$, all trial $x$, and all trial BPG points $\widehat y$ satisfying \eqref{eq:local-rs-test} lie in a fixed compact subset of $\Omega$.
    Choose a compact convex $\mathcal{K}\subset\Omega$ whose interior contains these points and $\{w\in\Omega:D_\psi(u,w)\leq D_\psi(u,y_0)\}$.
    The latter set is compact by condition~(iv), as in the proof of Lemma~\ref{lem:c5-compact-trials}.

    Use the trial notation of Algorithm~\ref{alg:strong-full}, so that $\theta=\theta_-+\mu\alpha=1+\mu\omega$ and $\kappa\alpha^2=2\lambda\omega\theta$.
    For a trial satisfying \eqref{eq:local-rs-test}, first define the constrained weighted mirror point
    \begin{equation*}
        \widehat z=\argmin_{w\in \mathcal{K}}
        \left\{\theta_-D_\psi(w,z)+\mu\alpha D_\psi(w,x)
        +\alpha\langle g_\lambda(x),w-z\rangle\right\}.
    \end{equation*}
    It exists uniquely by compactness, continuity, and strict convexity.
    Its variational inequality and the three-point identity give, for every $v\in \mathcal{K}$,
    \begin{align}
        &\alpha\langle g_\lambda(x),z-v\rangle
        +\theta D_\psi(v,\widehat z)-\theta_-D_\psi(v,z)
        -\mu\alpha D_\psi(v,x)\notag\\
        &\quad\leq\alpha\langle g_\lambda(x),z-\widehat z\rangle
        -\theta_-D_\psi(\widehat z,z)-\mu\alpha D_\psi(\widehat z,x)\notag\\
        &\quad\leq\frac{C_{\mathcal{K}}^2}{2\kappa}\frac{\theta}{\theta_-}
        \,\omega P_\lambda(x).
        \label{eq:sc-constrained-mirror}
    \end{align}
    Indeed, condition~(v), $D_\psi(x,\widehat y)\leq\lambda P_\lambda(x)$, and the coefficient equation bound the middle expression by
    \begin{equation*}
        C_{\mathcal{K}}\sqrt{\frac{2\omega\theta P_\lambda(x)D_\psi(\widehat z,z)}{\kappa}}
        -\theta_-D_\psi(\widehat z,z).
    \end{equation*}
    Completing the square gives \eqref{eq:sc-constrained-mirror}, exactly as in \eqref{eq:c5-constrained-mirror}, with the weights $\theta,\theta_-$.
    For a nonterminal trial, relative strong convexity and \eqref{eq:local-rs-test} imply $\mu\lambda\leq1$, and hence
    \begin{equation*}
        \frac{\mu\alpha}{\theta}
        =\sqrt{\frac{2\mu\lambda}{\kappa}\frac{\mu\omega}{\theta}}
        \leq\sqrt{\frac2\kappa}.
    \end{equation*}
    Therefore $\theta/\theta_-\leq2$ when $\kappa\geq8$, and the right-hand side of \eqref{eq:sc-constrained-mirror} is at most $\omega P_\lambda(x)$ when $\kappa\geq\max\{8,C_{\mathcal{K}}^2\}$.

    Taking $v=u$ in \eqref{eq:sc-constrained-mirror} and adding the objective estimate in the proof of Lemma~\ref{lem:strong-one-step} yields
    \begin{align*}
        \omega(\Phi(\widehat y)-\Phi^\star)+\theta D_\psi(u,\widehat z)
        &\leq\omega_-(\Phi(y)-\Phi^\star)+\theta_-D_\psi(u,z)
        -\mu\omega_-D_\psi(y,x)\\
        &\leq D_\psi(u,y_0).
    \end{align*}
    Nonnegativity of the objective gap and $\theta\geq1$ imply $\widehat z\in\interior \mathcal{K}$.
    The constraint is inactive, so optimality gives
    \begin{equation*}
        \nabla\psi(\widehat z)
        =\frac{\theta_-\nabla\psi(z)+\mu\alpha\nabla\psi(x)-\alpha g_\lambda(x)}{\theta}=p.
    \end{equation*}
    Legendre duality gives $p\in\interior\dom\psi^*$ and $\widehat z=\nabla\psi^*(p)\in\Omega$, verifying the algorithm's mirror update and conjugate-domain check.
    Taking $v=z$ in \eqref{eq:sc-constrained-mirror} gives
    \begin{equation*}
        \theta D_\psi(z,\widehat z)-\mu\alpha D_\psi(z,x)
        \leq\omega P_\lambda(x),
    \end{equation*}
    which implies \eqref{eq:strong-certificate} since $-\mu\omega_-D_\psi(y,x)\leq0$.

    The coefficient check precedes the BPG solve and must also be controlled.
    Every nonterminal accepted step has $\mu\lambda_k\leq1$, so all trial stepsizes satisfy
    \begin{equation*}
        \lambda\leq\gamma_+\max\{\lambda_0,1/\mu\}.
    \end{equation*}
    Thus $\kappa>2\mu\gamma_+\max\{\lambda_0,1/\mu\}$ also ensures the coefficient condition.
    Above a fixed threshold, no $\kappa$-dependent rejection can occur.
    Warm starts never increase $\kappa$, and each increase multiplies it by $\gamma_\kappa$, proving $\sup_k\kappa_k<\infty$.
    Proposition~\ref{prop:accepted-lambda-lower} bounds $\lambda_k$ away from zero, and Theorem~\ref{thm:strong-main} gives the claimed linear rate.
\end{proof}

\begin{remark}
For Euclidean geometry $\psi(x)=\|x\|^2/2$ and KL geometry $\psi(x)=\sum_i(x_i\log x_i-x_i)$, Theorem~\ref{thm:sc-acceleration-conditions} applies whenever a minimizer $u\in\Omega$ exists.
Indeed, $D_\psi(u,\cdot)$ is convex (for KL, its Hessian is $\operatorname{diag}(u_i/x_i^2)$), so (ii) holds with $\beta=0$ and (iii) is unnecessary; the conjugate domains are full dual spaces, and the positive definite Hessians imply (v).
This also covers negative entropy on the simplex, with all derivatives and interiors taken in its affine hull.
\end{remark}

\begin{remark}
For Burg geometry $\psi(x)=-\sum_i\log x_i$ on $\R_{++}^n$, relative strong convexity excludes finite objective values on the boundary and gives, for $\Phi(x)\leq\Phi(y_0)$,
    \begin{equation*}
    \sum_i\log\frac{(y_{0,i}+x_i)^2}{4y_{0,i}x_i}
    \leq\frac{2(\Phi(y_0)-\Phi^\star)}{\mu}
    \end{equation*}
by the midpoint inequality for $\Phi-\mu\psi$.
Each nonnegative term diverges as $x_i\downarrow0$ or $x_i\to\infty$, so the initial sublevel is compact in $\Omega$ and contains a unique minimizer $u$.
Moreover, $D_\psi(u,x)+\psi(x)=\sum_i u_i/x_i+\text{constant}$ is convex, giving (ii) with $\beta=1$; compactness means that (iii) holds automatically.
The open conjugate domain $\dom\psi^*=\R_{--}^n$ and positive definite Hessian verify the remaining conditions.
Thus Theorem~\ref{thm:sc-acceleration-conditions} guarantees linear convergence here as well.
\end{remark}

\section{Convergence analysis for nonconvex objectives}
\label{app:nonconvex}

We impose the common assumptions, including $\Phi^\star>-\infty$.
The analysis below uses no relative lower bound or DC decomposition and does not assume $\Omega=\R^n$.
Unlike in GA-BPGc and GA-BPGsc, the variable $z_k$ stores the trial point for the next iteration.
The fixed auxiliary geometry $\phi$ is used only to construct that point, not in the stationarity bound.

\subsection{Algorithm and well-definedness}
\label{app:nonconvex-algorithm}

\begin{algorithm}[!t]
    \caption{GA-BPGnc}
    \label{alg:nonconvex-full}
    \SetKwFunction{BT}{Backtracking}

    \KwIn{
        A Legendre function $\phi$;
        $y_0\in\Omega\cap\dom \rho$;
        $\lambda_{\max}\geq\lambda_0>0$;
        $\gamma_+\geq1$, $\gamma_->1$;
        $\sigma\in(0,1)$;
        a rule for choosing weights $\beta_j\geq0$.
    }

    $z_0\gets y_0$\;

    \For{$k=1,2,\ldots$}{
    \eIf{$k=1$}{
        $\lambda\gets\lambda_0$\;
    }{
        $\lambda\gets
            \min\{\gamma_+\lambda_{k-1},\lambda_{\max}\}$\;
    }

    Choose $\beta_{k+1}\geq0$ by the prescribed rule\;
    $(x_k,\widehat y_k,y_k,z_k,\lambda_k,\delta_k)\gets$
    \BT{$y_{k-1},z_{k-1},\lambda,\beta_{k+1}$}\;
    }

    \Procedure{\BT{$y,z,\lambda,\beta$}}{
        $x\gets z$
        \Comment*{Coupling}

        $\widehat y\in T_\lambda(x)$
        \Comment*{BPG update}

        \If{$D_f(\widehat y,x)>D_\psi(\widehat y,x)/\lambda$}{
            \Return{\BT{$y,z,\lambda/\gamma_-,\beta$}}\;
        }

        $\delta\gets\Phi(y)-\Phi(\widehat y)$\;

        \eIf{$\delta>0$}{
            $y_+\gets\widehat y$\;
        }{
            $y_+\gets y$\;
        }

        $z_+\gets y_+$\;

        \If{$\delta<\frac{\sigma\lambda}{2}\mathcal R_\lambda(x)$}{
            \Return{$(x,\widehat y,y_+,z_+,\lambda,\delta)$}\;
        }

        \If{$\{y_+,y\}\not\subset
                \interior\dom \phi$}{
            \Return{$(x,\widehat y,y_+,z_+,\lambda,\delta)$}\;
        }

        $p\gets
            \nabla\phi(y_+)
            +\beta[
                \nabla\phi(y_+)
                -\nabla\phi(y)
            ]$\;

        \If{$p\notin
            \interior\dom \phi^*$}{
            \Return{$(x,\widehat y,y_+,z_+,\lambda,\delta)$}\;
        }

        $\widehat z\gets\nabla\phi^*(p)$
        \Comment*{Mirror update}

        \If(\Comment*[f]{Criterion}){
            $\widehat z\in\Omega\cap\dom \rho$
        }{
            $z_+\gets\widehat z$\;
        }

        \Return{$(x,\widehat y,y_+,z_+,\lambda,\delta)$}\;
    }
\end{algorithm}

Backtracking holds $x_k$ fixed.
Every early return in the criterion completes the outer iteration with $z_k=y_k$.
Thus neither an invalid gradient nor an invalid inverse mirror map is evaluated.
Each weight $\beta_{k+1}$ is chosen from the history available before backtracking and held fixed throughout its trials.

\begin{proposition}[Well-defined outer iterations]
    \label{prop:nonconvex-well-defined}
    Every nonterminal outer iteration is completed after finitely many BPG trials.
    Its variables satisfy $x_k,\widehat y_k,y_k,z_k\in \Omega\cap\dom \rho$, and $\mathcal R_{\lambda_k}(x_k)>0$.
\end{proposition}

\begin{proof}
    Initially $z_0=y_0\in\Omega\cap\dom \rho$.
    Inductively, $x_k=z_{k-1}$ is feasible and interior.
    Solvability places every BPG solution in the same set, and Proposition~\ref{prop:local-rs-backtracking} gives finite backtracking.
    If the stopping test fails, Proposition~\ref{prop:bpg-model-gap} gives $\mathcal R_{\lambda_k}(x_k)>0$.
    The point $y_k$ is either $\widehat y_k$ or $y_{k-1}$.
    The default $z_k=y_k$ is replaced only by a candidate that passes the explicit feasibility test, closing the induction.
\end{proof}

\subsection{Descent and reset}
\label{app:nonconvex-descent}

\begin{proof}[Proof of~\cref{lem:nonconvex-descent}]
    The acceptance rule immediately gives~\eqref{eq:nonconvex-descent}.
    For $x_k=y_{k-1}$, inequality~\eqref{eq:gap-descent} gives $\delta_k\geq\frac{\lambda_k}{2}\mathcal R_{\lambda_k}(x_k)>0$, hence $\tau_k\geq1$.
    If $\tau_k<\sigma$, then $\delta_k<\frac{\sigma\lambda_k}{2}\mathcal R_{\lambda_k}(x_k)$, so the algorithm keeps $z_k=y_k$.
    Consequently $x_{k+1}=y_k$, and the preceding argument applies to the next iteration.
    An invalid extrapolation also gives a safe next iteration, although this additional case is not needed for the bound.
\end{proof}

\subsection{Stationarity bound in terms of objective decreases and stepsizes}
\label{app:nonconvex-posterior}

\begin{proof}[Proof of the first inequality
        in~\eqref{eq:nonconvex-posterior}] The first iteration is safe because $x_1=z_0=y_0$.
    If it is nonterminal, Lemma~\ref{lem:nonconvex-descent} gives $\tau_1\geq1$, so $\mathcal A_N$ is nonempty.
    Telescope~\eqref{eq:nonconvex-descent}:
    \begin{equation*}
        \Phi(y_0)-\Phi(y_N)
        =
        \frac12
        \sum_{k\in\mathcal A_N}
        \tau_k\lambda_k
        \mathcal R_{\lambda_k}(x_k).
    \end{equation*}
    All weights in this sum are positive.
    Since $y_N\in \cl\Omega$, its left side is at most $\Phi(y_0)-\Phi^\star$.
    Taking the minimum residual outside the sum and dividing by the sum of weights proves the first bound in~\eqref{eq:nonconvex-posterior}.
\end{proof}

\subsection{Worst-case bound}
\label{app:nonconvex-worstcase}

\begin{lemma}[Two-iteration inequality]
    \label{lem:nonconvex-pair}
    If both iterations $k$ and $k+1$ are nonterminal, then
    \begin{equation*}
        [\tau_k]_++[\tau_{k+1}]_+\geq\sigma.
    \end{equation*}
\end{lemma}

\begin{proof}
    If $\tau_k\geq\sigma$, its first term suffices.
    Otherwise, the next iteration is safe and $\tau_{k+1}\geq1>\sigma$ by Lemma~\ref{lem:nonconvex-descent}.
\end{proof}

\begin{proof}[Proof of the second inequality
        in~\eqref{eq:nonconvex-posterior}] Summing Lemma~\ref{lem:nonconvex-pair} over $k=1,\ldots,N-1$ gives
    \begin{equation*}
        2\sum_{k=1}^N[\tau_k]_+\geq\sigma(N-1).
    \end{equation*}
    The stepsize lower bound therefore implies
    \begin{equation*}
        \sum_{k\in\mathcal A_N}\tau_k\lambda_k
        \geq
        \underline\lambda\sum_{k=1}^N[\tau_k]_+
        \geq
        \tfrac12\sigma\underline\lambda(N-1).
    \end{equation*}
    Substitute this lower bound into the first inequality in~\eqref{eq:nonconvex-posterior}.
    A residual at most $\varepsilon^2$ thus requires at most
    \begin{equation*}
        1+
        \left\lceil
            \frac{4(\Phi(y_0)-\Phi^\star)}
            {\sigma\underline\lambda\varepsilon^2}
        \right\rceil
    \end{equation*}
    outer iterations, unless stationarity is returned earlier.
\end{proof}

\begin{corollary}[A sharper counting bound]
    \label{cor:nonconvex-sharp}
    With $c_\sigma=\min\{\sigma,1/2\}$, after any $N\geq1$ nonterminal iterations,
    \begin{equation*}
        \sum_{k=1}^N[\tau_k]_+\geq c_\sigma N,
        \qquad
        \min_{k\in\mathcal A_N}
        \mathcal R_{\lambda_k}(x_k)
        \leq
        \frac{2(\Phi(y_0)-\Phi^\star)}
        {c_\sigma\underline\lambda N}.
    \end{equation*}
\end{corollary}

\begin{proof}
    Call an iteration bad when $\tau_k<\sigma$.
    The first iteration is not bad, and each bad iteration other than the last is followed by a safe iteration with $\tau\geq1$.
    Pair each such bad iteration with its successor.
    If the last iteration is bad, pair it with the first iteration instead.
    No pairs overlap, since bad iterations cannot be consecutive and the first iteration is safe.

    If there are $b$ bad iterations, each of the $b$ pairs contributes at least one to the sum, and each of the remaining $N-2b$ iterations contributes at least $\sigma$.
    Therefore the sum is at least
    \begin{equation*}
        b+\sigma(N-2b)
        \geq
        \min\{\sigma,1/2\}N.
    \end{equation*}
    Combine this with the stepsize lower bound and the first inequality in~\eqref{eq:nonconvex-posterior}.
\end{proof}

\subsection{Number of BPG solves}
\label{app:nonconvex-cost}

\begin{corollary}[Amortized backtracking cost]
    \label{cor:nonconvex-cost}
    Through $N$ nonterminal outer iterations,
    \begin{equation*}
        N_{\mathrm{BPG}}(N)
        \leq
        N+
        \frac{
            (N-1)\log\gamma_+
            +\log(\lambda_0/\underline\lambda)
        }{\log\gamma_-}.
    \end{equation*}
    Consequently, the BPG-solve complexity for $\mathcal R_{\lambda_k}(x_k)\leq\varepsilon^2$ is $\O(\varepsilon^{-2})$.
\end{corollary}

\begin{proof}
    Each outer iteration performs one successful upper-bound trial and $m_k$ failed trials, so
    \begin{equation*}
        N_{\mathrm{BPG}}(N)=N+\sum_{k=1}^Nm_k.
    \end{equation*}
    Apply~\eqref{eq:total-backtracking} and the outer-iteration bound.
    A terminal iteration satisfies the same backtracking estimate with its actual number of trials.
    Constructing the next mirror candidate requires no additional BPG subproblem solve.
\end{proof}

\subsection{Relation to norm-based stationarity}
\label{app:nonconvex-stationarity}

We use $\dist(x,S)\coloneq\inf_{u\in S}\|x-u\|$ for the Euclidean distance from $x$ to a set $S$.

\begin{proposition}[Subgradient comparison]
    \label{prop:nonconvex-subgradient}
    Let
    \begin{equation*}
        s_k=
        g_{\lambda_k}(x_k)
        +\nabla f(\widehat y_k)-\nabla f(x_k).
    \end{equation*}
    Then $s_k\in\nabla f(\widehat y_k)+\partial\rho(\widehat y_k)$.
    If a uniform constant $\Gamma>0$ satisfies $\|s_k\|^2\leq \Gamma\mathcal R_{\lambda_k}(x_k)$ for $k\in\mathcal A_N$, then
    \begin{equation*}
        \min_{k\in\mathcal A_N}
        \dist^2
        (
        0,\nabla f(\widehat y_k)+\partial\rho(\widehat y_k)
        )
        \leq
        \frac{4\Gamma(\Phi(y_0)-\Phi^\star)}
        {\sigma\underline\lambda(N-1)},
        \qquad N\geq2.
    \end{equation*}
\end{proposition}

\begin{proof}
    Membership follows from Lemma~\ref{lem:bpg-subgradient}.
    The squared distance is at most $\|s_k\|^2$.
    Apply the assumed uniform comparison and the second inequality in~\eqref{eq:nonconvex-posterior}.
\end{proof}

The uniform comparison with $\Gamma$ is an additional assumption beyond upper relative smoothness.

\begin{corollary}[Euclidean residual]
    \label{cor:nonconvex-euclidean}
    If $\psi=\|\cdot\|^2/2$ and $\Omega=\R^n$, then, for every $x\in\dom\rho$ and $y\in T_\lambda(x)$,
    \begin{equation*}
        \left\|
            \frac{x-y}{\lambda}
        \right\|^2
        \leq
        \mathcal R_\lambda(x).
    \end{equation*}
    If also $\rho=0$, then $\mathcal R_\lambda(x)=\|\nabla f(x)\|^2$.
    Thus the nonconvex main-text bounds recover the corresponding bounds on the squared gradient norm in the smooth unconstrained case.
\end{corollary}

\begin{proof}
    The first assertion follows from~\eqref{eq:gap-lower} with $D_\psi(x,y)=\|x-y\|^2/2$.
    When $\rho=0$, $T_\lambda(x)=\{x-\lambda\nabla f(x)\}$, and direct substitution gives $\mathcal R_\lambda(x)=\|\nabla f(x)\|^2$.
\end{proof}

The stationarity bounds concern interior evaluation points.
Convergence of the full sequence and stationarity at boundary accumulation points require separate analysis.

\section{Additional numerical experiments and more discussions}
\label{app:additional-experiments-discussion}

\subsection{Additional experiments}
\label{app:additional-experiments}

\paragraph{Auxiliary geometry and stepsize.}
Using the convex Poisson setup on NYTimes from Section~\ref{sec:experiments}, we compare GA-BPGc with negative-entropy, Burg, and squared-Euclidean auxiliary geometries, while keeping the Burg BPG geometry $\psi$ fixed.
For each auxiliary geometry, Figure~\ref{fig:additional-objective-gap} compares adaptive proximal stepsizes with fixed $\lambda=1/L$.
Negative entropy reaches the smallest objective gaps under both stepsize rules, and adaptation further improves convergence in iterations and wall-clock time.
Burg yields smaller reductions.

\begin{figure}[t]
    \centering
    \includegraphics[width=0.95\linewidth]{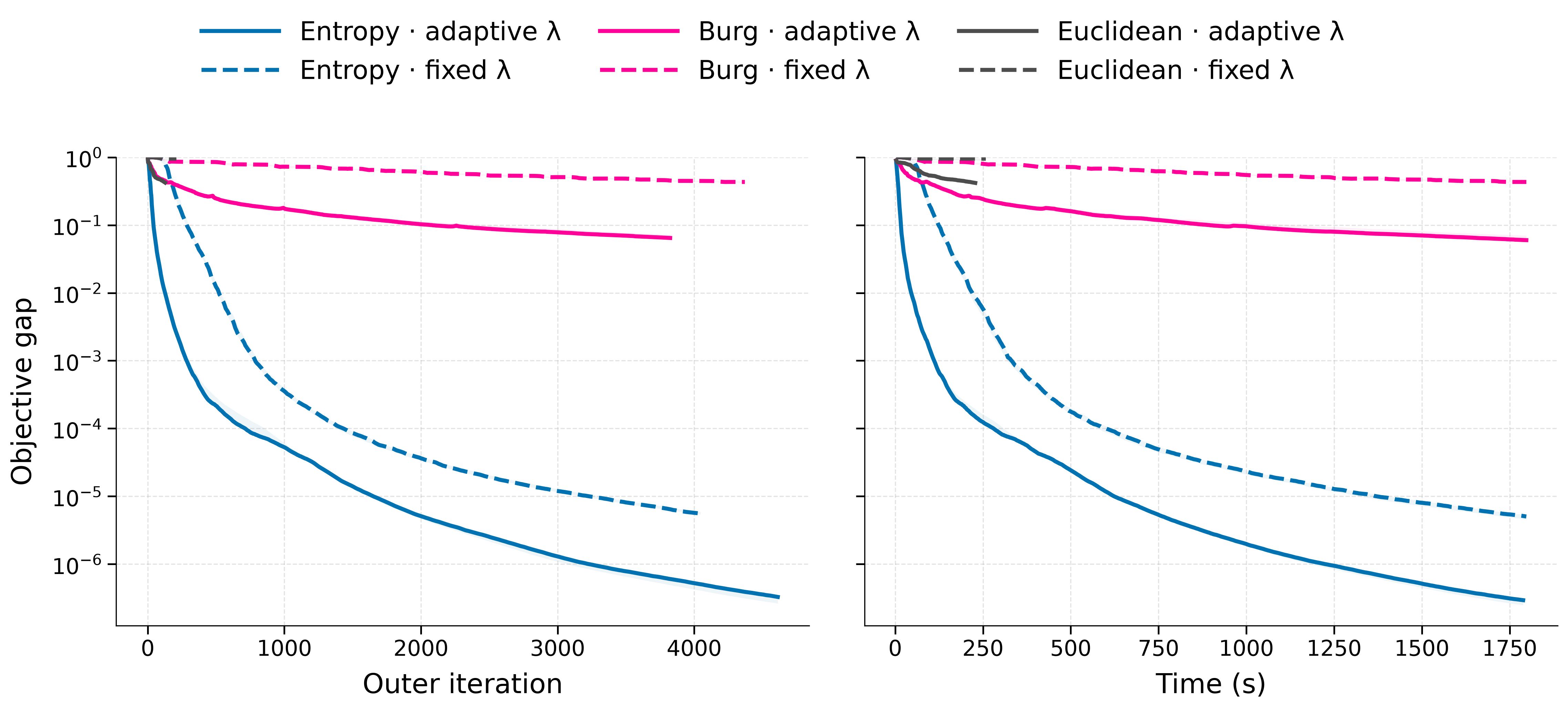}
    \caption{Log plot of the normalized objective gap for GA-BPGc settings comparison on NYTimes.}
    \label{fig:additional-objective-gap}
\end{figure}

\paragraph{Convex D-optimal design.}
Following \citet{hanzely2021accelerated}, we solve $\min_{x\in\Delta_n}\Phi(x)\coloneq-\log\det(H\Diag(x)H^\top)$, where $\Diag(x)$ is the diagonal matrix with diagonal $x$ and $\Delta_n=\{x\in\R^n:x\geq0,\ \mathbf{1}^\top x=1\}$.
The columns of $H\in\R^{m\times n}$ are the raw LIBSVM \texttt{abalone} feature vectors ($m=8$, $n=4{,}177$).
All methods use the Burg geometry $\psi(x)=-\sum_i\log x_i$ and start at $x_0=\mathbf{1}/n$; GA-BPGc uses the auxiliary negative entropy $\phi(x)=\sum_i(x_i\log x_i-x_i)$.
Figure~\ref{fig:convex-abalone} compares GA-BPGc, BPG-LS, and ABPG-g over five timing runs, each capped at $5{,}000$ iterations or $1{,}800$ seconds, showing medians and interquartile ranges.
The deterministic iteration traces coincide across repetitions.
Separate long runs of BPG, BPG-LS, and ABPG-g provide $x_{\mathrm{ref}}$ and the certified lower bound $\Phi_{\mathrm{LB}}=\Phi(x_{\mathrm{ref}})-\langle\nabla\Phi(x_{\mathrm{ref}}),x_{\mathrm{ref}}\rangle+\min_j[\nabla\Phi(x_{\mathrm{ref}})]_j$.
We use $\Phi_{\mathrm{ref}}=\Phi_{\mathrm{LB}}$.
The dotted curve uses the right-hand side of~\eqref{eq:convex-posterior} with $u=x_{\mathrm{ref}}$, adding $\Phi(x_{\mathrm{ref}})-\Phi_{\mathrm{LB}}$ before applying the same normalization.

\begin{figure}[t]
    \centering
    \includegraphics[width=0.9\linewidth]{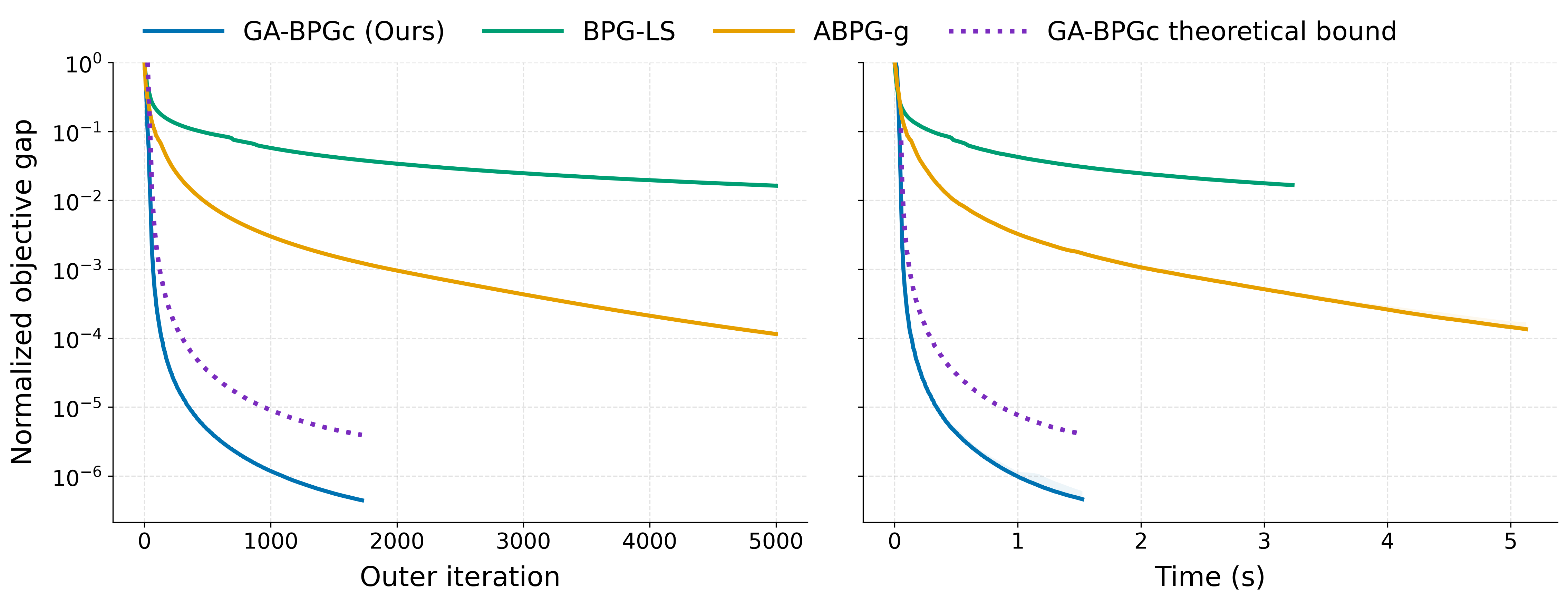}
    \caption{Log plot of the normalized objective gap for D-optimal design on the \texttt{abalone} data.}
    \label{fig:convex-abalone}
\end{figure}

\paragraph{Nonconvex phase retrieval.}
We consider the normalized $\ell_1$-regularized phase-retrieval problem $\min_{x\in\R^d}\Psi(x)\coloneqq\frac{1}{4m}\sum_{r=1}^{m}\left(\langle a_r,x\rangle^2-b_r\right)^2+\frac{1}{m}\|x\|_1.$
We generate five independent instances (five seeds) with $d=1000$, $m=6000$, independent $a_r\sim\mathcal{N}(0,I)$, and noiseless measurements $b_r=\langle a_r,x^\star\rangle^2$ of a $5\%$-sparse Gaussian signal.
All algorithms share each instance's spectral initializer $x_0$ \citep{candes2015phase}.
We compare GA-BPGnc with BPG, BPGe, CoCaIn-BPG, and BPDCAe~\citep{takahashi2022new}.
All algorithms use $\psi(x)=\phi(x)=\frac{1}{4}\|x\|^4+\frac{1}{2}\|x\|^2$.
BPG and BPGe use $\lambda=0.1/L$ with a Gaussian-based estimate $L$; BPDCAe uses the deterministic bound \citep{takahashi2022new}.
GA-BPGnc and BPDCAe use $\beta_k=(t_{k-1}-1)/t_k$ and $t_k=(1+(1+4t_{k-1}^2)^{1/2})/2$.
GA-BPGnc and CoCaIn-BPG use their respective backtracking rules.
Runs are capped at $2{,}000$ iterations; only GA-BPGnc stops early upon meeting its stationarity criterion.
Figure~\ref{fig:phase-retrieval} plots normalized empirical objective gaps relative to $\Psi_{\mathrm{ref}}$.
Across all five instances, GA-BPGnc stops after $61$--$97$ iterations, reaching the objective gap faster than all baselines in iterations and wall-clock time.

\begin{figure}[!tp]
    \centering
    \includegraphics[width=0.9\linewidth]{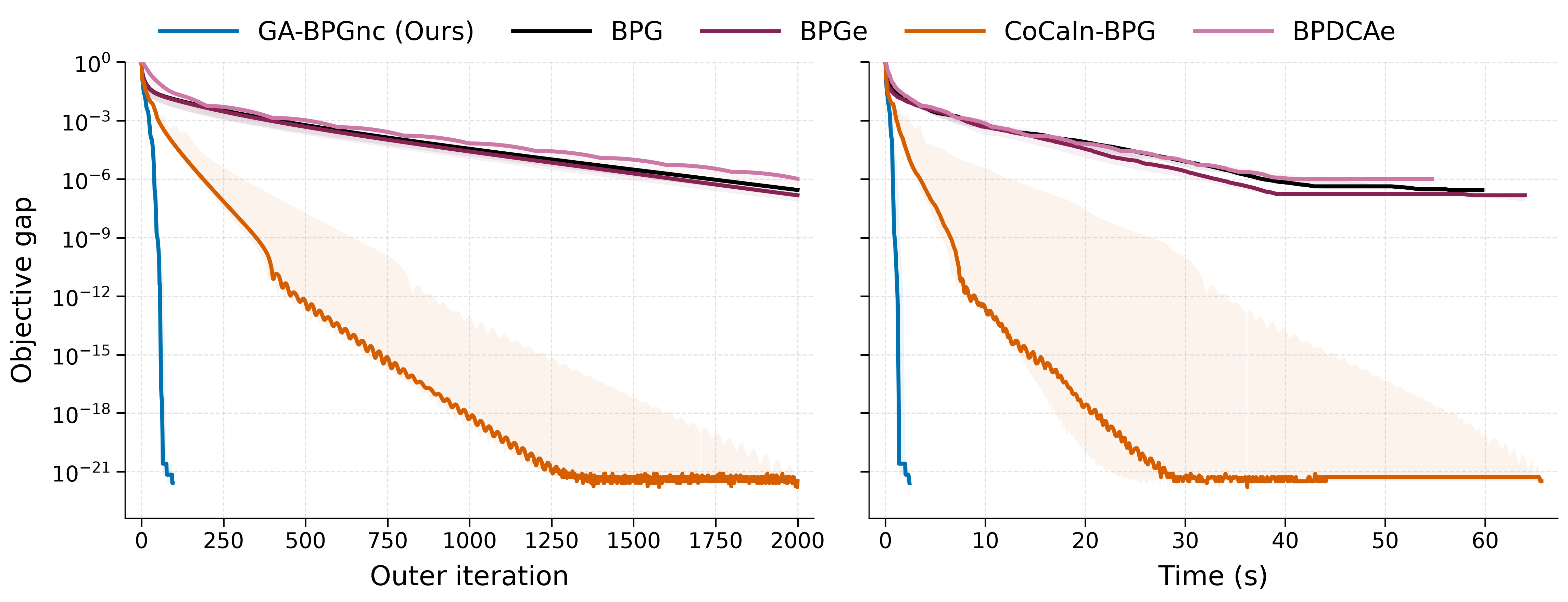}
    \caption{Log plot of the normalized objective gap for nonconvex $\ell_1$-regularized phase retrieval.}
    \label{fig:phase-retrieval}
\end{figure}

\paragraph{Nonconvex linear inverse problem.}
We solve $\min_{x\in\R^n}\Phi(x)\coloneq \sum_{i=1}^{m}\log\!\left(1+((Ax-b)_i)^2\right)+10^{-3}\|x\|_1$, where $A\in\R^{m\times n}$ has independent $\mathcal{N}(0,1/m)$ entries ($m=1000$, $n=2000$) and $b=Ax^\star+\varepsilon$.
The vector $x^\star$ has $10\%$ nonzero entries drawn independently from $\mathcal{N}(0,50^2)$, and $\varepsilon$ is Gaussian noise with standard deviation $10^{-3}$.
The smooth term has a $2\|A\|_2^2$-Lipschitz gradient, where $\|A\|_2$ denotes the spectral norm.
Figure~\ref{fig:nonconvex-linear-inverse} compares GA-BPGnc, BPG, BPGe, and CoCaIn-BPG using Euclidean geometry and $2{,}000$ outer iterations from $x_0=0$.
We report medians and interquartile ranges of the normalized empirical objective gap over five independently generated $(A,x^\star,\varepsilon)$ instances; $\Phi_{\mathrm{ref}}$ is each instance's minimum objective across $5{,}000$-iteration runs of all four methods.

\begin{figure}[t]
    \centering
    \includegraphics[width=0.9\linewidth]{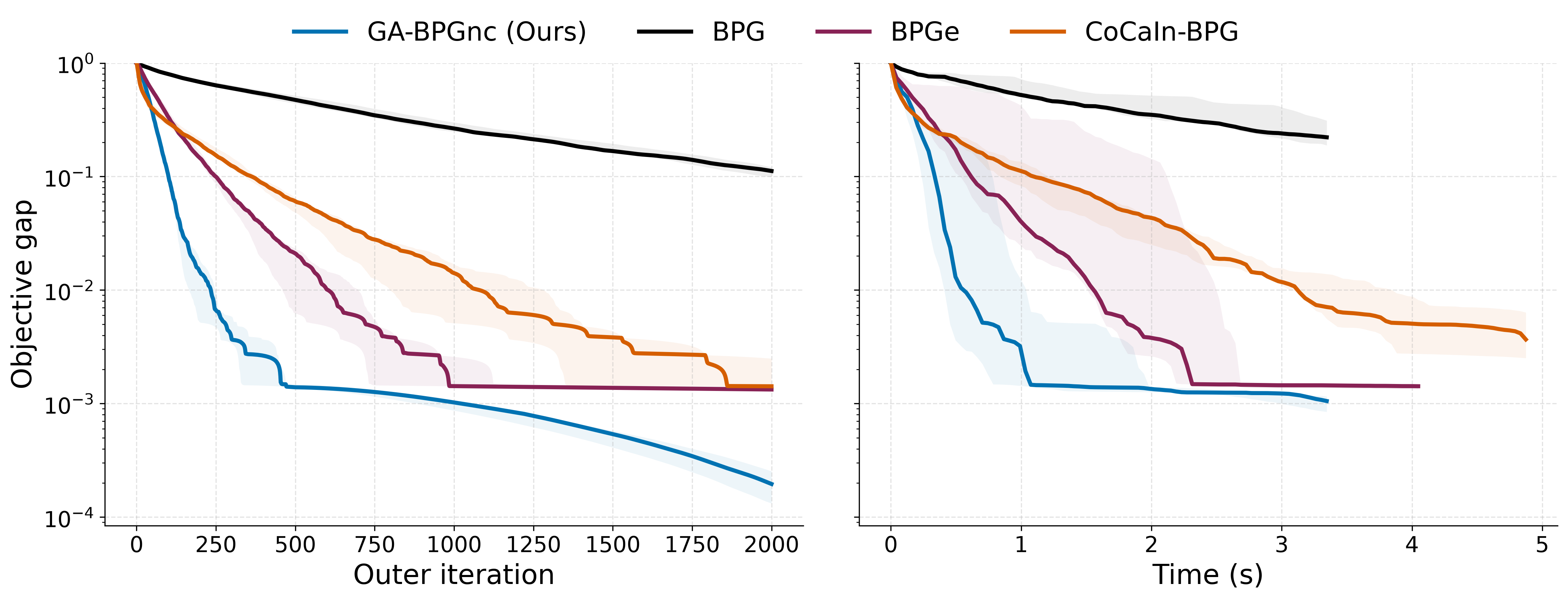}
    \caption{Log plot of the normalized objective gap for the nonconvex linear inverse problem.}
    \label{fig:nonconvex-linear-inverse}
\end{figure}

\subsection{Discussion}
\label{app:discussions}
The auxiliary-geometry experiment in Figure~\ref{fig:additional-objective-gap} is consistent with the sufficient conditions in Table~\ref{tab:convex-kernel-conditions} and Theorem~\ref{thm:convex-kernel-certificate}.
For Burg--KL geometry, bounded initial objective sublevels and an interior minimizer are enough for acceleration (Corollary~\ref{cor:c5-burg-kl}).
For the Burg--Burg geometries, the certificate uses the stronger condition that the initial sublevel be compact in $\Omega$; boundedness alone does not exclude approach to its boundary.
The Euclidean auxiliary generator has $\interior\dom\phi=\R^n\neq\Omega$ and falls outside this theorem.
Its repeated domain-check rejections and large $\kappa$ also explain the slow progress: large $\kappa$ reduces the accepted ratios $\lambda_k/\kappa_k$ that control the convex convergence bound~\eqref{eq:convex-posterior}.

For D-optimal design, GA-BPGc converges faster than BPG-LS and ABPG-g in both iterations and wall-clock time (Figure~\ref{fig:convex-abalone}).
The simplex guarantees bounded initial sublevels, as required by Corollary~\ref{cor:c5-burg-kl}.
Although the existence of an interior minimizer is not guaranteed, since zero weights can coexist with a positive definite matrix, Figure~\ref{fig:convex-abalone} shows fast empirical convergence of GA-BPGc.

\end{document}